\documentclass[12pt,reqno]{amsart}
\usepackage{amsmath,amsthm,amssymb,amscd,url,enumerate}
\usepackage{graphicx}
\usepackage{afterpage}
\usepackage{tikz}
\usepackage[all]{xy}
\usepackage[colorlinks=true]{hyperref}
\usepackage[width=5.3in,height=8.5in, centering]{geometry}

\newcommand{\TITLE}{The sensual Apollonian circle packing}
\newcommand{\TITLERUNNING}{The sensual Apollonian circle packing}
\newcommand{\DATE}{\today}
\newcommand{\VERSION}{1}

\theoremstyle{plain} 
\newtheorem{theorem}{Theorem} 

\newtheorem{proposition}[theorem]{Proposition}
\newtheorem{lemma}[theorem]{Lemma}
\newtheorem{corollary}[theorem]{Corollary}

\theoremstyle{definition}
\newtheorem{definition}[theorem]{Definition}

\theoremstyle{remark}

\newtheorem*{acknowledgement}{Acknowledgements}

\numberwithin{theorem}{section}

  {\end{list}}

  {\end{list}}

\newcommand{\tightoverset}[2]{%
  \mathop{#2}\limits^{\vbox to -.5ex{\kern-1.05ex\hbox{$#1$}\vss}}}

\newcommand{\CC}{\mathbb{C}}

\newcommand{\PP}{\mathbb{P}}
\newcommand{\MM}{\mathbb{M}}
\newcommand{\QQ}{\mathbb{Q}}
\newcommand{\RR}{\mathbb{R}}
\newcommand{\ZZ}{\mathbb{Z}}

\renewcommand{\gcd}{{\operatorname{gcd}}}

\newcommand{\MOD}[1]{~(\textup{mod}~#1)}
\renewcommand{\pmod}{\MOD}

\renewcommand{\setminus}{\smallsetminus}
\renewcommand{\Re}{\operatorname{Re}}
\renewcommand{\Im}{\operatorname{Im}}

\newcommand\PSL{\operatorname{PSL}}
\newcommand\PGL{\operatorname{PGL}}
\newcommand\PGLZ{\PGL_2(\ZZ[i])}
\newcommand\SO{\operatorname{SO}^+_{1,3}(\ZZ)}
\newcommand\SOR{\operatorname{SO}^+_{1,3}(\RR)}

\newcommand\Soder{MR1178634}
\newcommand\Hirst{MR0209981}
\newcommand\Northshield{MR2250043}
\newcommand\BourgainFuchs{MR2813334}
\newcommand\FuchsBulletin{MR3020827}
\newcommand\AhSt{MR1466797}
\newcommand\ntone{MR1971245}
\newcommand\gttwo{MR2183489}
\newcommand\gtone{MR2173929}
\newcommand\bestsavin{MR2914631}
\newcommand\bk{1205.4416}

\newcommand\coxeter{MR0230204}
\newcommand\boyd{MR658230}
\newcommand\Conway{MR1478672}
\newcommand\beyond{MR1903421}
\newcommand\Sage{Sage}
\newcommand\BFuchs{MR2813334}

\title[\TITLERUNNING]{\vspace*{-1.3cm} \TITLE}

\author{Katherine E. Stange}

\date{\DATE, Draft \#\VERSION}
\address{%
Department of Mathematics, University of Colorado,
Campux Box 395, Boulder, Colorado 80309-0395}
\email{kstange@math.colorado.edu}
\keywords{Apollonian circle packings, projective linear group, topograph, Hermitian form, M\"obius transformation}
\subjclass[2010]{Primary: 52C26, 11E39, Secondary: 11E12, 11E16}

\thanks{The author's research has been supported by NSF MSPRF 0802915.
}

\begin{document}


\begin{abstract}
        The curvatures of the circles in integral Apollonian circle packings, named for Apollonius of Perga (262-190 BC), form an infinite collection of integers whose Diophantine properties have recently seen a surge in interest.  Here, we give a new description of Apollonian circle packings built upon the study of the collection of bases of $\ZZ[i]^2$, inspired by, and intimately related to, the `sensual quadratic form' of Conway.
\end{abstract}

\maketitle

\tableofcontents

\section{Introduction}

In their delightful monograph \emph{The Sensual Quadratic Form} \cite{\Conway}, Conway and Fung draw the following picture of $\PP^1(\QQ)$.

\begin{center}
\includegraphics[width=4in]{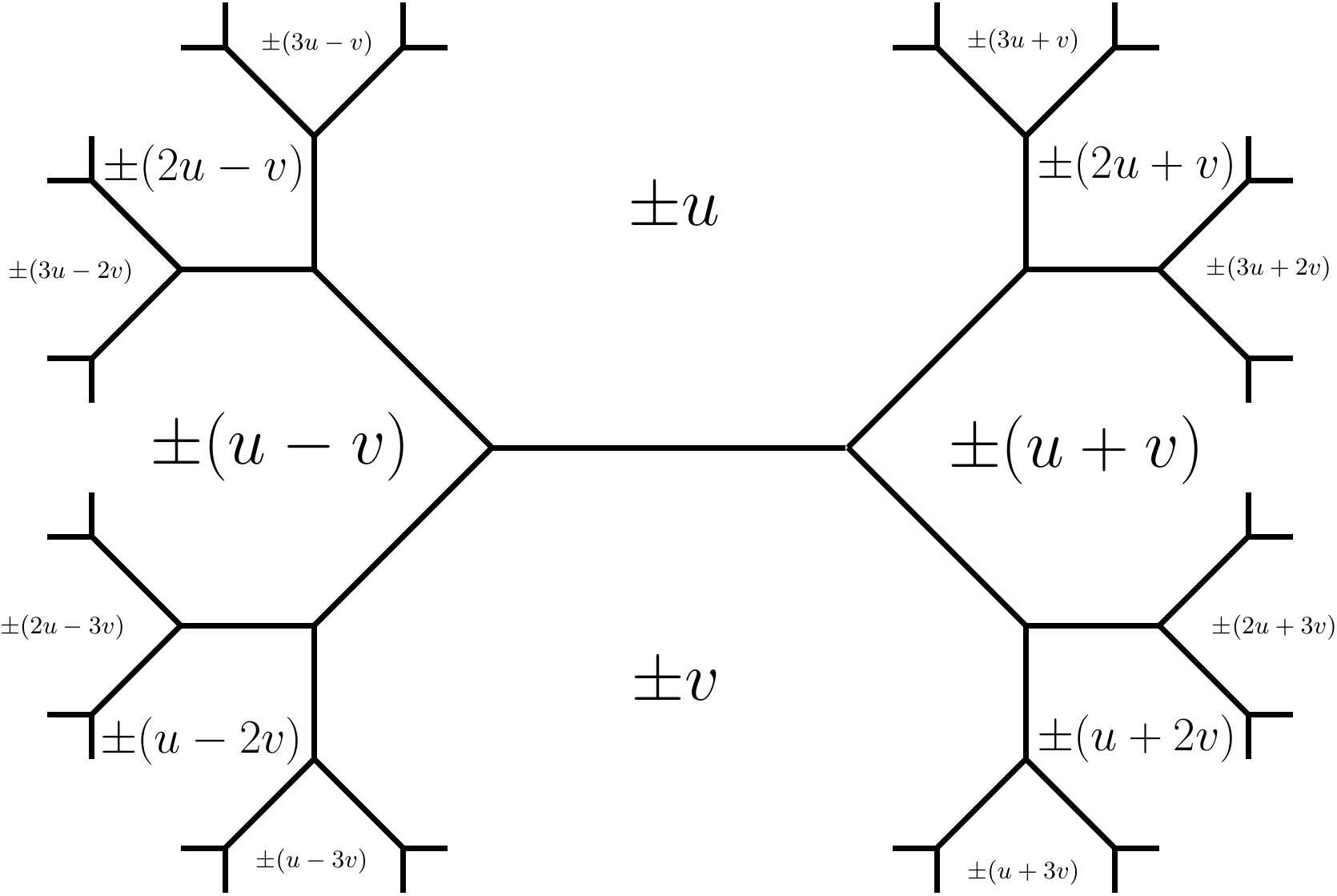}
\end{center}



This is but a small piece of an infinite froth of planar regions traced out by an infinite tree of valence three.  In the picture, $u$ and $v$ refer to any particular $\ZZ$-basis of $\ZZ^2$.  Each region is labelled with a what they termed a \emph{lax vector}, i.e. a primitive vector of $\ZZ^2$ (that is, having coprime coordinates) considered up to sign.  These lax vectors represent the elements of $\PP^1(\QQ)$.

A \emph{basis} of lax vectors is a pair such that, once signs are chosen, one obtains a $\ZZ$-basis in the usual sense.  A \emph{superbasis} is a triple of lax vectors, any two of which form a basis.  The branching graph, called the \emph{topograph}, can be defined as the graph whose edges correspond to bases of lax vectors, and whose vertices correspond to \emph{superbases}.  Vertices touch edges where superbases contain bases; an edge connects two vertices because any basis is contained in exactly two possible superbases.  That is, $\{ \pm u, \pm v \}$ is contained in $\{ \pm u, \pm v, \pm (u+v)\}$ and $\{ \pm u, \pm v, \pm (u-v) \}$.  

One then shows that the graph is a single infinite valence-three tree which can be made planar in such a way that it breaks up the plane into infinitely many infinite regions, the boundary of each consisting of exactly those edges and vertices containing a particular lax vector: we label the region with this vector to obtain the picture above.  For example, the central edge represents the basis $\{ \pm u, \pm v \}$ since it separates the regions labelled $\pm u$ and $\pm v$, and the vertex to its right represents the superbasis $\{ \pm u, \pm v, \pm(u+v) \}$.

Just as a linear form is determined by its values on a basis, a binary quadratic form is determined by its values on a superbasis.  For forms on $\ZZ^2$, the topograph provides lovely visual evidence of this fact.  Evaluate the quadratic form $f$ on the lax vector labelling each region.  The parallelogram law,
\[
        f(x+y) + f(x-y) = 2 f(x) + 2 f(y),
\]
relates the values on the four regions surrounding any edge of the topograph.  Hence knowing the values surrounding any one vertex allows one, iteratively, to deduce every other value in the topograph.  Imagining the values of $f$ as altitude, Conway classifies integral quadratic forms (as indefinite, positive definite, and so on) by their topographical terrain, describing their lakes, wells, weirs and river valleys.

The following equally beautiful picture is named for the geometer Apollonius of Perga (262-190 BC).

\begin{center}
\includegraphics[width=3.5in]{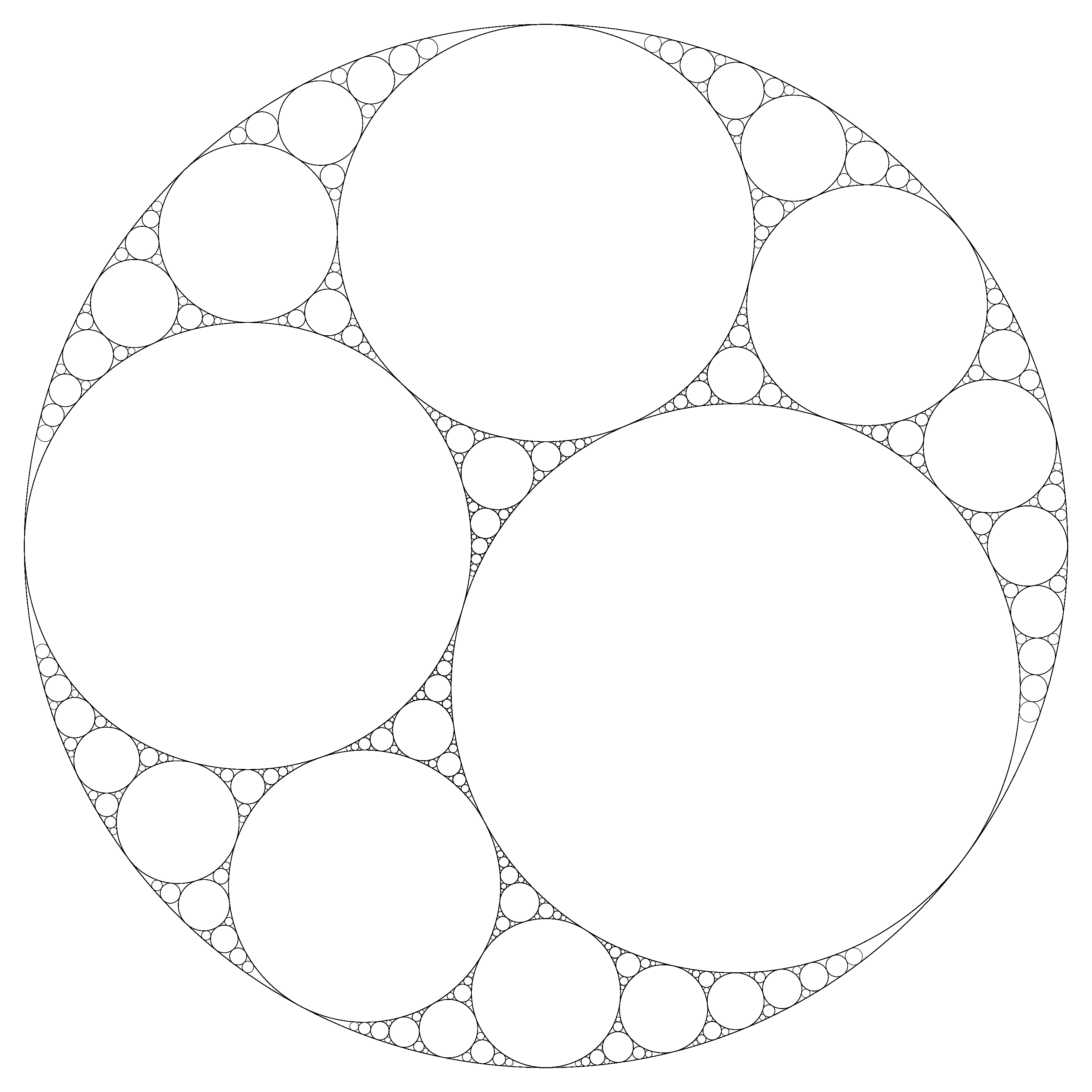}
\end{center}

To form such an \emph{Apollonian circle packing}, one starts with three mutually tangent circles of disjoint interiors (sometimes the ``interior'' of a circle may be defined to be the outside).  Given such a triple, there are exactly two ways to \emph{complete} it to a collection of four mutually tangent circles, again of disjoint interiors, which is called a \emph{Descartes quadruple}.  In the following picture, a triple is drawn in solid lines, and the two possible completions are dotted.  The interior of the large circle is the outside.

\begin{center}
\includegraphics[width=1.5in]{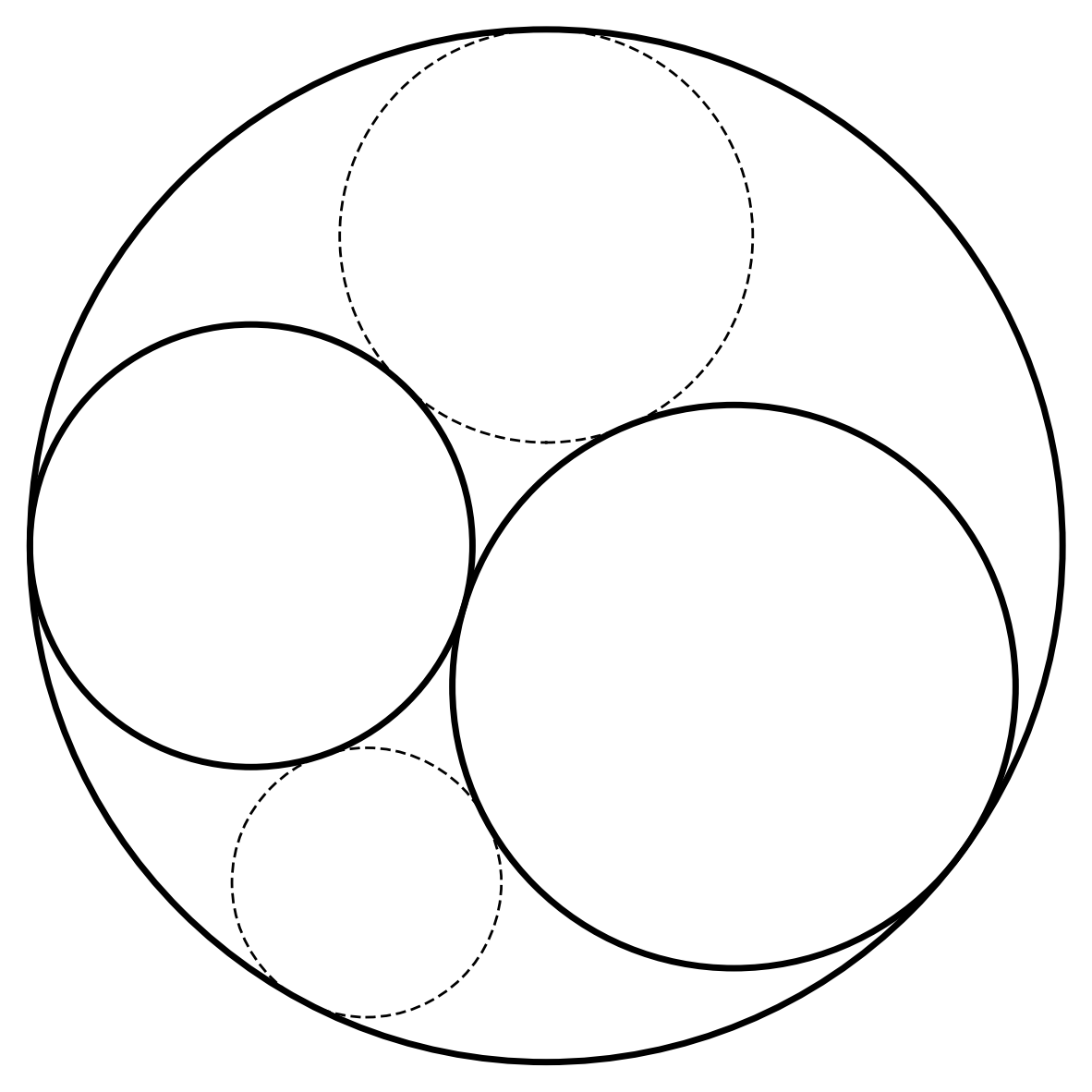}
\end{center}

Start with any three mutually tangent circles, and add in the missing completions, thereby growing the set of circles to five.  Repeat this process ad infinitum, at each stage adding \emph{all} the missing completions of \emph{all} mutually tangent triples in the collection.  

\begin{center}
        \raisebox{-.49\height}{\includegraphics[width=1in]{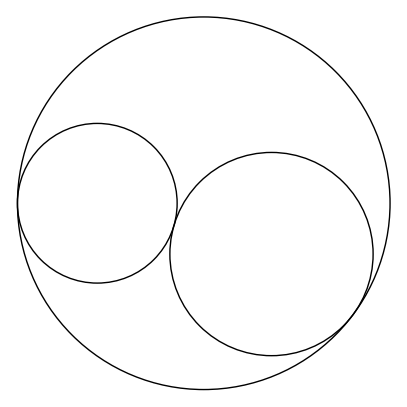}
\includegraphics[width=1in]{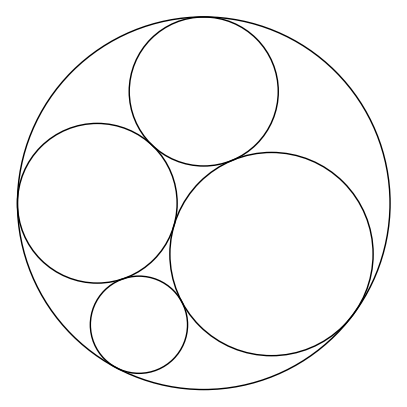}
\includegraphics[width=1in]{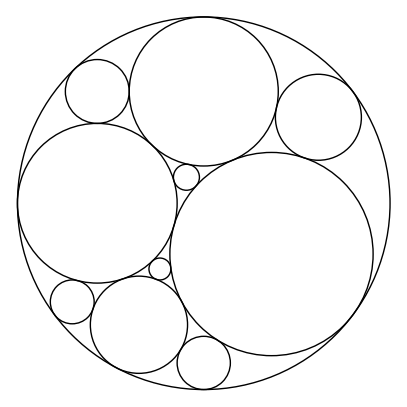}}
$\;\ldots\;$
\raisebox{-.49\height}{\includegraphics[width=1in]{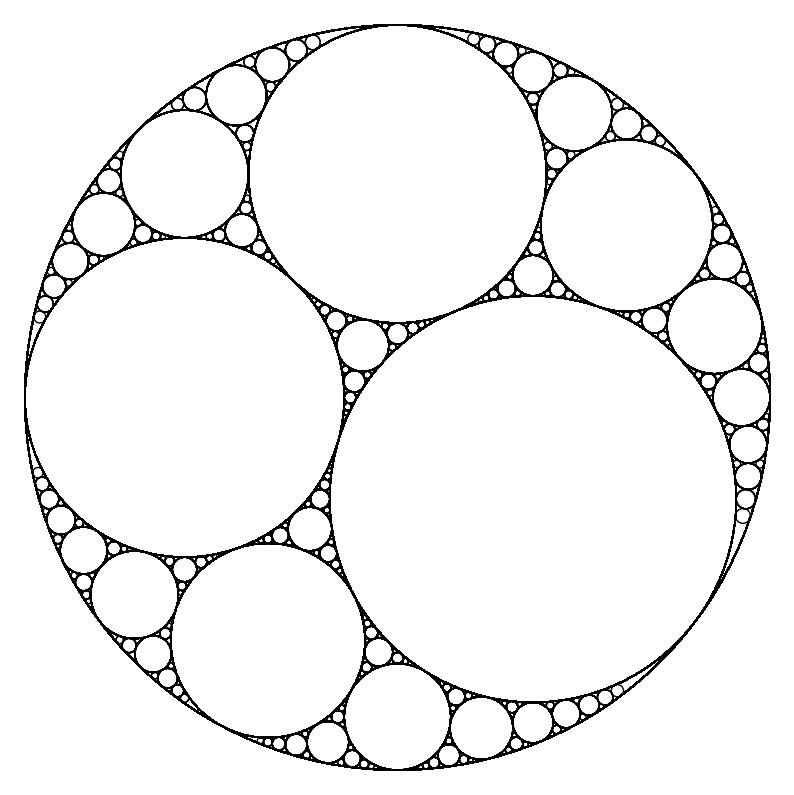}}
\end{center}

The resulting collection of circles is an Apollonian circle packing.  
The interest for the number theorist lies in the \emph{quadratic Descartes relation} between the curvatures (inverse radii) $a,b,c,d$ of four circles in a Descartes quadruple:
\begin{equation}
        \label{eqn:descquad}
        2( a^2 + b^2 + c^2 + d^2) = (a+b+c+d)^2.
\end{equation}
(For circles whose ``interior is outside,'' curvature is negative.)
This remarkable observation has been traced by Pedoe \cite{MR0215169} to Descartes, in a letter to Princess Elisabeth of Bohemia \cite[p.49]{Descartes}.
Considering the trace of \eqref{eqn:descquad} as a quadratic equation in $d$
entails that the two possible curvatures $d$ and $d'$ which complete a triple of curvatures $a,b,c$ to a Descartes quadruple satisfy the \emph{linear Descartes relation}:
\begin{equation}
        d+d' = 2(a+b+c).
        \label{eqn:desclinear}
\end{equation}
In particular, any Apollonian circle packing seeded with a Descartes quadruple of \emph{integer} curvatures $a,b,c,d$ consists entirely of circles of integer curvature!  Such a packing is called \emph{integral}.  Furthermore, it is \emph{primitive} if its curvatures have no common factor.  An example is given in Figure \ref{fig:6eleven}.

\begin{figure}[ht!]
        \centering
        \includegraphics[width=5in]{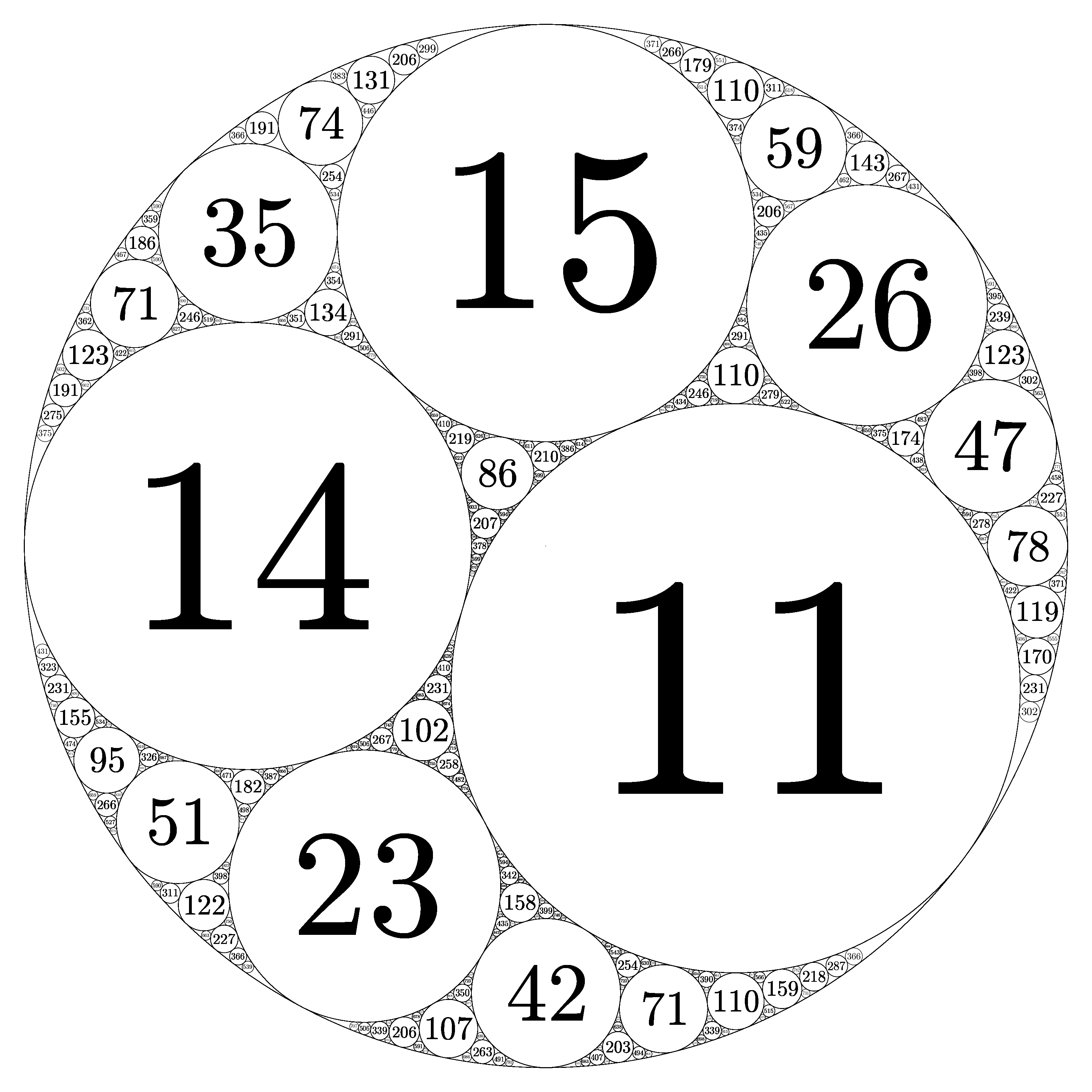}
        \caption{A primitive integral Apollonian circle packing in which each circle is labelled with its curvature.  The outer circle has curvature $-6$.}
        \label{fig:6eleven}
\end{figure}

The pheromones emanated by such an ambrosial arrangement of integers have inevitably attracted number theorists, with some remarkable results.  See \cite{MR2813334, 1205.4416, \ntone, MR2784325, MR2800340} for theorems on the density of the set in the integers, the occurrence of prime curvatures, and much more.

Perhaps the reader has, at this point, already noticed a similarity in the stories of Conway and Apollonius.  In the first, each basis (pair of vectors) can be completed to a superbasis (triple of vectors) in exactly two ways, and the values of $f$ on the two new possibilities sum to twice the sum on the original pair:
\[
        f(u+v) + f(u-v) = 2( f(u) + f(v) ).
\]
In the second, each triple of tangent circles can be completed to a Descartes quadruple in exactly two ways, and the curvatures of the two new possibilities sum to twice the sum on the original triple:
\begin{equation*}
        d + d' = 2( a + b + c).
\end{equation*}

Is this more than a coincidence?  Conway has drawn a picture of $\PP^1(\QQ)$.  In what follows, we will draw a picture of $\PP^1(\QQ(i))$.  It will be called the \emph{Apollonian kingdom}, in anticipation of a connection to the geometer's tale.  The kingdom is a graph whose edges correspond to the elements of $\PP^1(\QQ(i))$, in other words, the primitive lax vectors for $\ZZ[i]^2$ (vectors of coprime entries, considered up to unit multiple).  Imagining it in three dimensions, we will see a grand complex of \emph{chambers} (complete graphs on 4 vertices, i.e.\ tetrahedra) sharing \emph{walls} (complete graphs on 3 vertices, i.e.\ triangles).

Our main theorem is that the whole great pride of Apollonian circle packings lives in this kingdom:  it is a kingdom of palaces, each the opulent home of one primitive integral Apollonian circle packing.




\begin{theorem}[Main Bijection, Introductory Version]
        \label{thm:mbintro}
We have the following correspondences:

\vspace{1em}
\begin{tabular}{l|l|l}
         Apollonian kingdom &  $\PP^1(\QQ(i))$ meaning & Circle packing meaning \\
        \hline
        \hline
        vertex & lax lattice & Gaussian circle \\
        \hline
          edge & lax vector & tangency within a \\
           & & Descartes quadruple  \\
           & & of Gaussian circles  \\
        \hline
         triangle & superbasis & three tangent circles within a\\
         (wall) & & Descartes quadruple\\
         & & of Gaussian circles\\
        \hline
   tetrahedron & ultrabasis & Descartes quadruple \\
   (chamber) & & of Gaussian circles \\
        \hline
                    component &  & Apollonian circle packing \\
                     (palace) &  & of Gaussian circles 
\end{tabular}
\end{theorem}

The first and last terms in the second column will be explained in the next section.  

We will consider packings in their native habitat:  the extended complex plane, $\CC_\infty$.  

\begin{definition}
        A \emph{Gaussian circle} is the image of $\RR$ under a M\"obius transformation having coefficients in $\ZZ[i]$.
\end{definition}

The complete collection of Gaussian circles is dense in $\CC_\infty$, but if we draw only those circles with bounded curvature, as in Figure \ref{fig:superpacking}, we can see the ornate structure in their arrangement.  They collect themselves into nested and intertwined Apollonian circle packings!  In fact, all primitive integral Apollonian circle packings can be realized using Gaussian circles, after a contraction by a factor of two; see Section \ref{sec:after}.

\begin{figure}[ht!]
        \centering
        \includegraphics[width=4in]{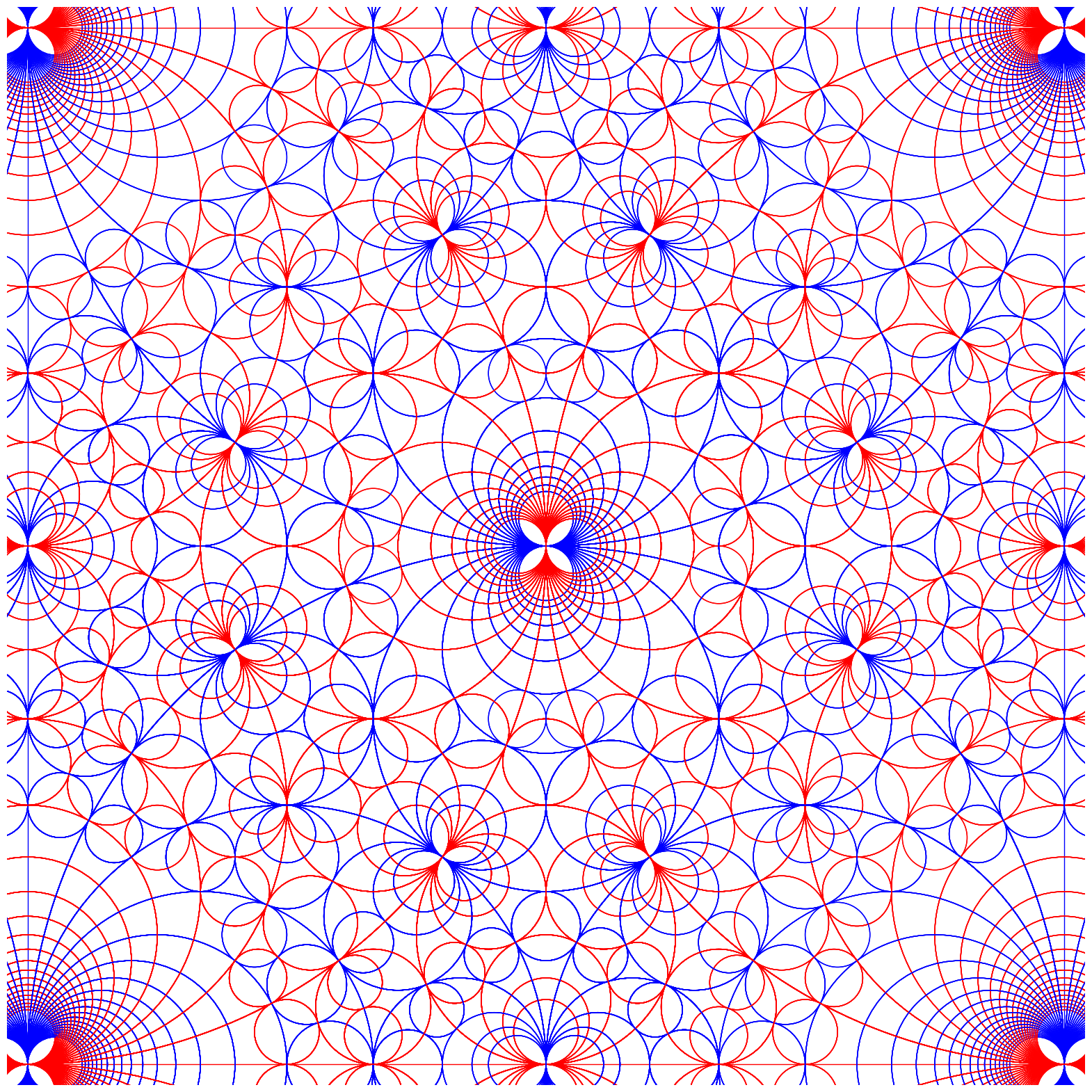}
        \caption{All Gaussian circles of curvature less than 40 within a square with side-length 1, centred on $\frac{1+i}{2}$.  See Section \ref{sec:after}.}
        \label{fig:superpacking}
\end{figure}

And their curvatures?  Hermitian forms take the place of the quadratic forms in Conway's story.  The proof of the Main Bijection relies on a second result:

\begin{theorem}
        \label{thm:herm}
        Let $H$ represent the imaginary part of a Hermitian form.  Then, $H$ takes a well-defined value on vertices of the Apollonian kingdom in such a way that the values on two adjacent chambers satisfy a `linear Descartes relation': i.e.\ the labels on their \emph{common vertices}, doubled, add up to the sum of the labels of their \emph{non-common vertices}:
\begin{center}
\includegraphics[width=3.5in]{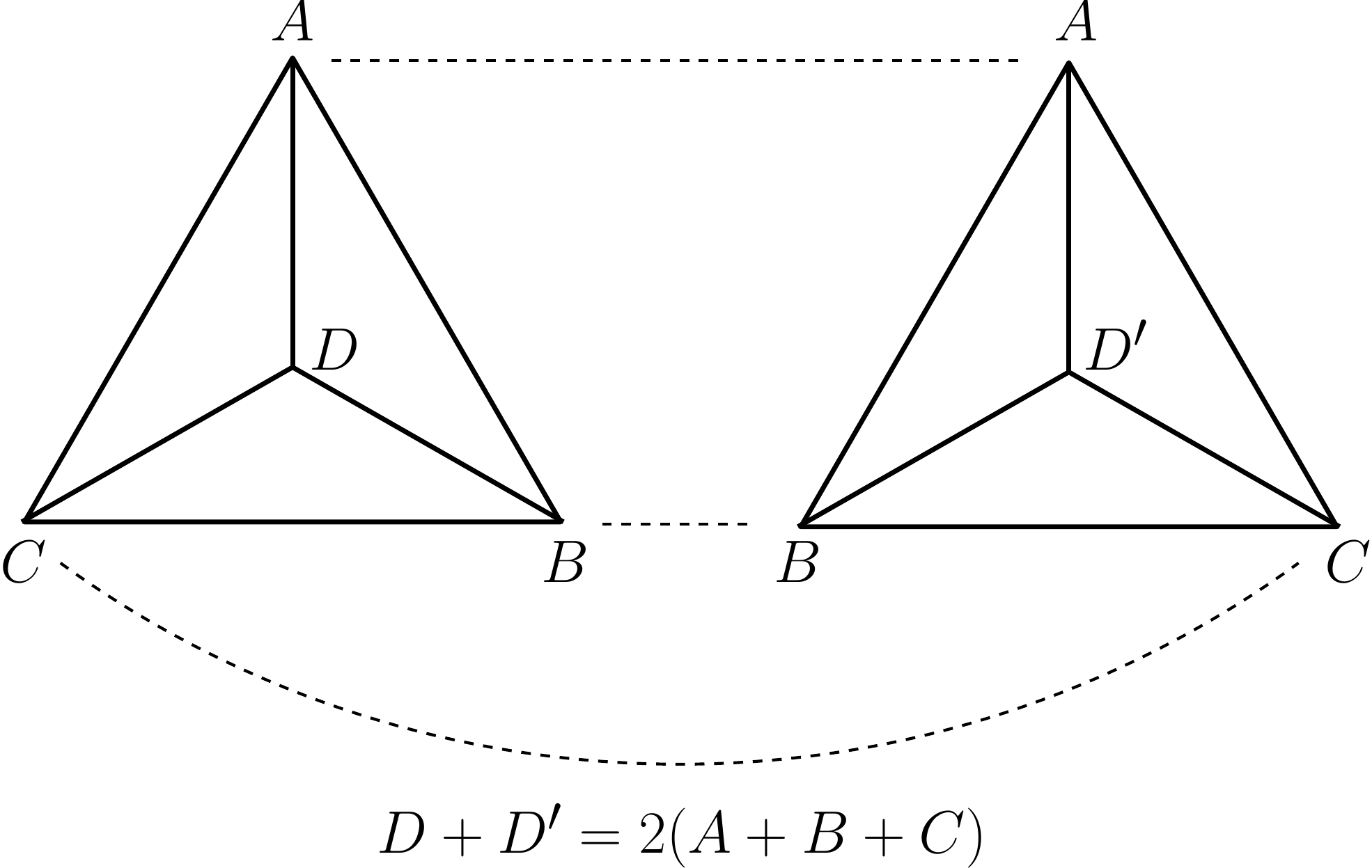}
\end{center}
\end{theorem}

In particular, for a special choice of $H$, these values are exactly the curvatures of the corresponding circles.  Also along the path to the Main Bijection, we'll find that the answer to our question goes beyond an analogy:  in fact, Conway's topograph is repeated infinitely often inside the Apollonian kingdom, once for each circle.  In a subsequent paper, this relationship allows us to rediscover the curvatures in a packing as values of quadratic forms (this connection was first observed in \cite{\ntone, \Northshield, SarnakLetter} and used to great effect in \cite{\BourgainFuchs,\bk}; for an exposition see \cite{\FuchsBulletin}).

\begin{acknowledgement}
        The author would like to thank Lionel Levine for drawing her attention to Apollonian circle packings, and would especially like to thank David Wilson for sharing data he had collected associating circles and lattices in $\ZZ^2$ in the context of abelian sandpiles.  It was an examination of his data that led to this study.  For more on this connection, see \cite{Lionel,Lionel2}.  
        The author owes a debt of gratitude to the invaluable and detailed suggestions provided by Jeffrey Lagarias, Andrew Granville, Jonathan Wise and an anonymous referee on an earlier draft.  The images in this paper were produced using IPE \cite{IPE} and Sage Mathematics Software \cite{\Sage}.
\end{acknowledgement}


\section{A kingdom of palaces: the statement of the Main Bijection}

We will approach the main bijection by defining two graphs:  the first is an \emph{Apollonian palace} for each Apollonian circle packing, and the second is the \emph{Apollonian kingdom} defined in the language of lax vectors of $\ZZ[i]^2$ (without reference to circles).  Then we will show that the Apollonian kingdom has infinitely many components, each the Apollonian palace for one of the infinitely many Apollonian circle packings in $\CC_\infty$.

\subsection{Palaces of Conway and Apollonius}
\label{newsec:apppal}

If one considers the regions of Conway and Fung's topograph -- the elements of $\PP^1(\QQ)$ -- as that which is essential to their story, one may redraw their picture in a new way:   define a vertex for each element of $\PP^1(\QQ)$, connecting them by an edge whenever, taken as a pair of primitive vectors of $\ZZ^2$, they form a basis.  This new picture can be drawn as an overlay on the original topograph as in Figure \ref{fig:conwaytriangle}.  Edges represent bases, and triangles represent superbases.  We will call this new picture \emph{Conway's palace} to distinguish it from the topograph.

\begin{definition}
\emph{Conway's palace} is the graph having one vertex for each lax vector of $\ZZ^2$, and having an edge joining each pair of vertices whose lax vectors form a basis.
\end{definition}

\begin{figure}
        \centering
\includegraphics[height=2.8in]{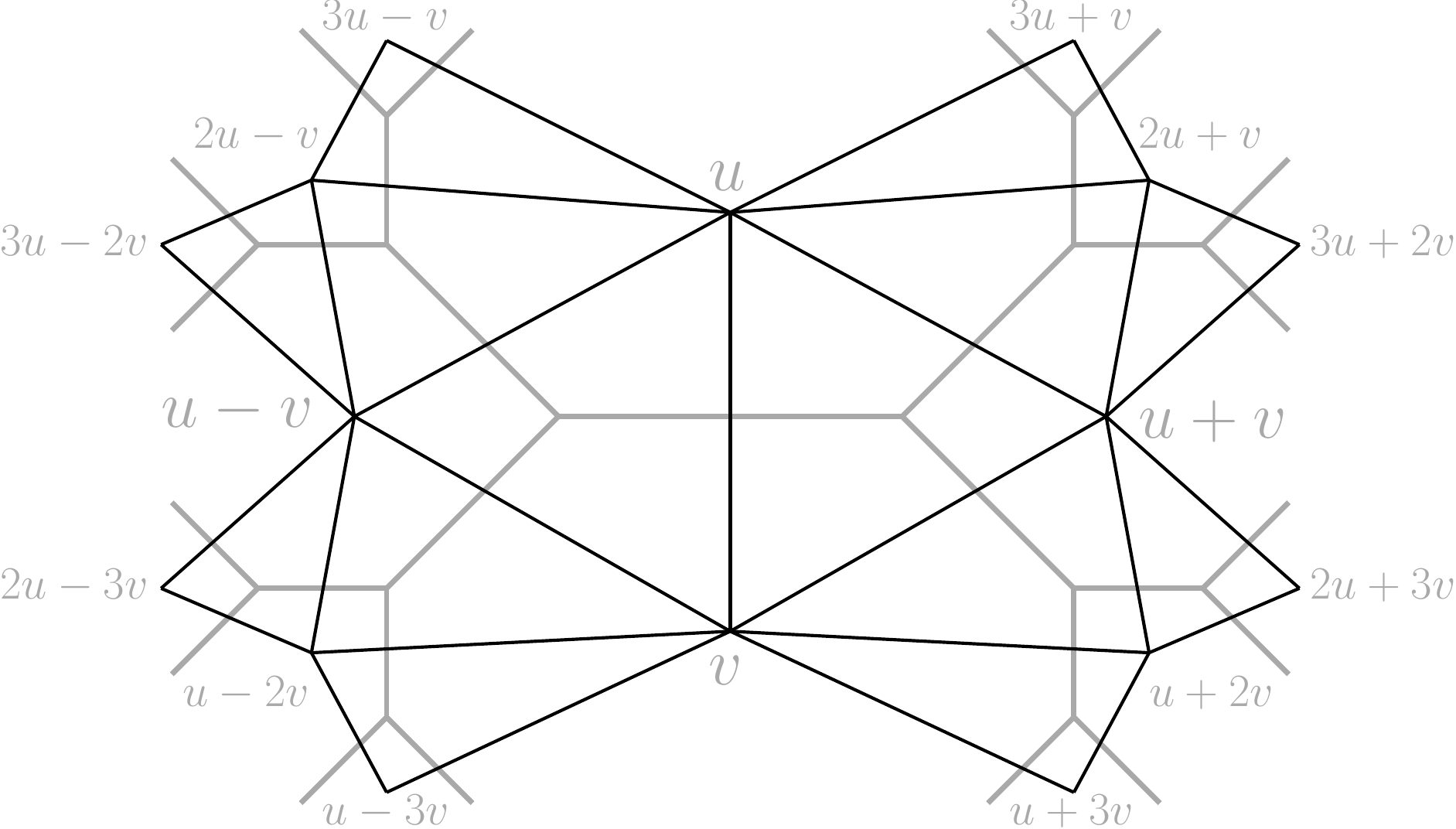}
\caption{A portion of Conway's topograph (grey), overlaid by the corresponding portion of Conway's palace (black).  
}
        \label{fig:conwaytriangle}
\end{figure}

Under our motivating analogy, the elements of $\PP^1(\QQ)$ in Conway's story become the circles in Apollonius'.  If we mark a vertex at the center of every circle in an Apollonian circle packing, and draw a line to connect every tangent pair, we obtain a picture not unlike Figure \ref{fig:conwaytriangle}, this one shown in Figure \ref{fig:kingdom}.  Triangles represent triples of tangent circles, and the \emph{tetrahedra} (i.e.\ complete subgraphs on 4 vertices) represent Descartes quadruples.  We will call this an \emph{Apollonian palace}.
\begin{figure}
        \centering
\includegraphics[height=3.2in]{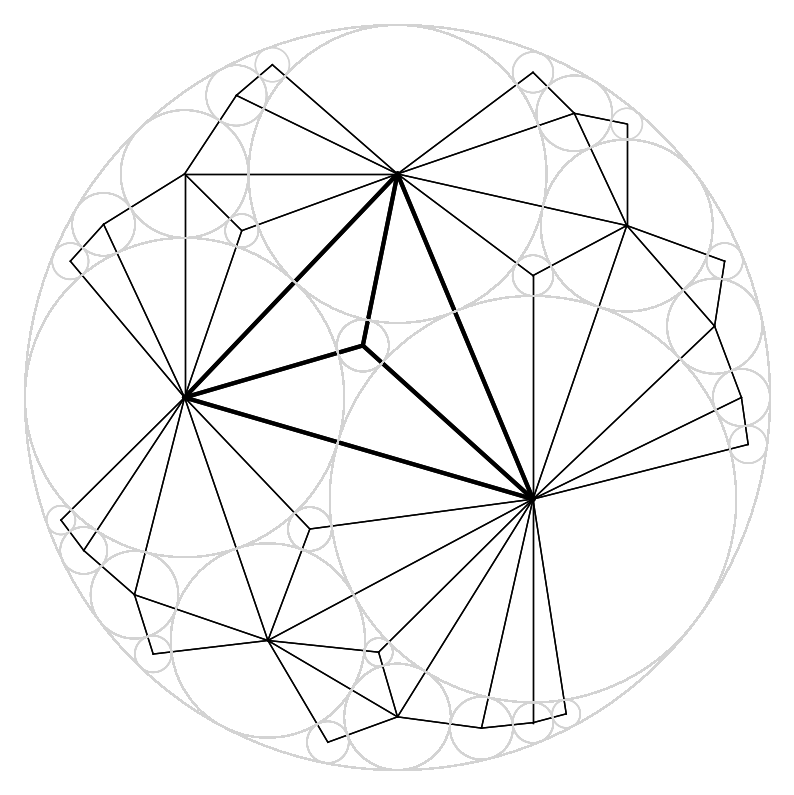}
        \caption{A portion of an Apollonian circle packing (grey) overlaid by the corresponding portion of its Apollonian palace (black).  Only circles whose curvature lies below a certain bound are shown; the vertex for the outer circle is not included.  An example tetrahedron is shown in thicker lines.}
        \label{fig:kingdom}
\end{figure}

\begin{definition}
The \emph{Apollonian palace} of an Apollonian circle packing is the graph having one vertex for each circle of the packing, and having an edge joining each pair of vertices whose circles are tangent.
\end{definition}



\subsection{Apollonian palaces in the complex plane}

To give this palace a bit more structure, we look at packings in the extended complex plane $\CC_\infty = \CC \cup \{\infty \}$.
This is a model of $\PP^1(\CC)$, wherein we associate to each one-dimensional subspace of $\CC^2$ generated by a vector $(z_1, z_2)$ the point $z_1/z_2$ in the extended complex plane.  We can think of $\PP^1(\QQ(i))$ as a subset of $\PP^1(\CC)$; it consists of those subspaces with slopes in $\QQ(i)$.  In the plane, these are the points $\QQ(i)_\infty = \QQ(i) \cup \{\infty\}$.

For a commutative ring $R$ with field of fractions $K$, the matrix group 
\[
        \PGL_2(R) = \left. \left\{ \begin{pmatrix}a & b \\ c & d \end{pmatrix} : ad-bc \in R^*; a,b,c,d \in R \right\} \right/ R^* 
\]
acts as automorphisms on $\PP^1(K)$ by matrix multiplication.  The corresponding action of $\PGL_2(\CC)$ on $\CC_\infty$ is via \emph{M\"obius transformations}, i.e.\
\[
        \begin{pmatrix} a & b \\ c & d \end{pmatrix} \cdot z \mapsto \frac{az+b}{cz+d}.
\]
The group $\PGL_2(\ZZ[i])$ acts the same way on $\QQ(i)_\infty$.  Because two matrices represent the same element of $\PGL_2(\ZZ[i])$ if and only if they represent the same element of $\PGL_2(\CC)$, we can consider $\PGL_2(\ZZ[i])$ a subgroup of $\PGL_2(\CC)$.

We define circles in $\CC_\infty$ to include not only the usual suspects, i.e.\ circles in the Euclidean metric on $\CC$, but also straight lines (these are `circles through $\infty$').  Then the M\"obius transformations take circles to circles.  In fact, three points of $\CC_\infty$ uniquely determine a circle.  Since the collection of M\"obius transformations is freely parametrized by the images of $0$, $1$ and $\infty$, any circle can be mapped to any other by choosing a transformation which takes three points on the former to three on the latter.

Recall that four oriented circles are in Descartes configuration (i.e.\ form a Descartes quadruple) if and only if they are mutually tangent of disjoint interiors.  To discuss interiors, we define \emph{orientation}, which is an additional binary datum associated to a circle:  we say the orientation is \emph{clockwise} or \emph{counter-clockwise}, indicating a direction of travel around the circle.  
By convention, the region to the right as one travels around the circle is its interior.  We also use the convention that \emph{clockwise} orientation of $\RR$ means travel to the right.  

M\"obius transformations can be considered to act on oriented circles in the following way.  Choose three points of a circle ordered according to the direction of travel; the orientation of the image circle is the direction of travel necessary to visit the three respective image points in the same order.  The continuity of M\"obius transformations on $\CC_\infty$ guarantees that this is well-defined.  

Figure \ref{fig:strip} shows the images of the real line $\RR$ under the following group of M\"obius transformations of $\CC_\infty$:
\[
        \left\langle 
        \begin{pmatrix} 1 & 0 \\ 1 & 1 \end{pmatrix},
                     \begin{pmatrix} 0 & i \\ i & 1 \end{pmatrix}
        \right\rangle
\]

One obtains an Apollonian circle packing seeded by a Descartes quadruple of curvatures $0$, $0$ (the two straight lines), $2$ and $2$ (this is demonstrated in Section \ref{sec:after}).  Figure \ref{fig:1223nohalf} shows another Apollonian circle packing in $\CC_\infty$: this one is the image of Figure \ref{fig:strip} under the transformation 
$\begin{pmatrix}i & 0 \\ 1 & i \end{pmatrix}$.  
Since M\"obius transformations preserve circles and tangencies, we can generate many Apollonian packings this way.

\newsavebox{\strip}
\savebox{\strip}{
$\left\langle \begin{pmatrix} 1 & 0 \\ 1 & 1 \end{pmatrix}, \begin{pmatrix} 0 & i \\ i & 1 \end{pmatrix}\right\rangle$}

\begin{figure}[ht!]
        \centering
        \includegraphics[width=4.7in]{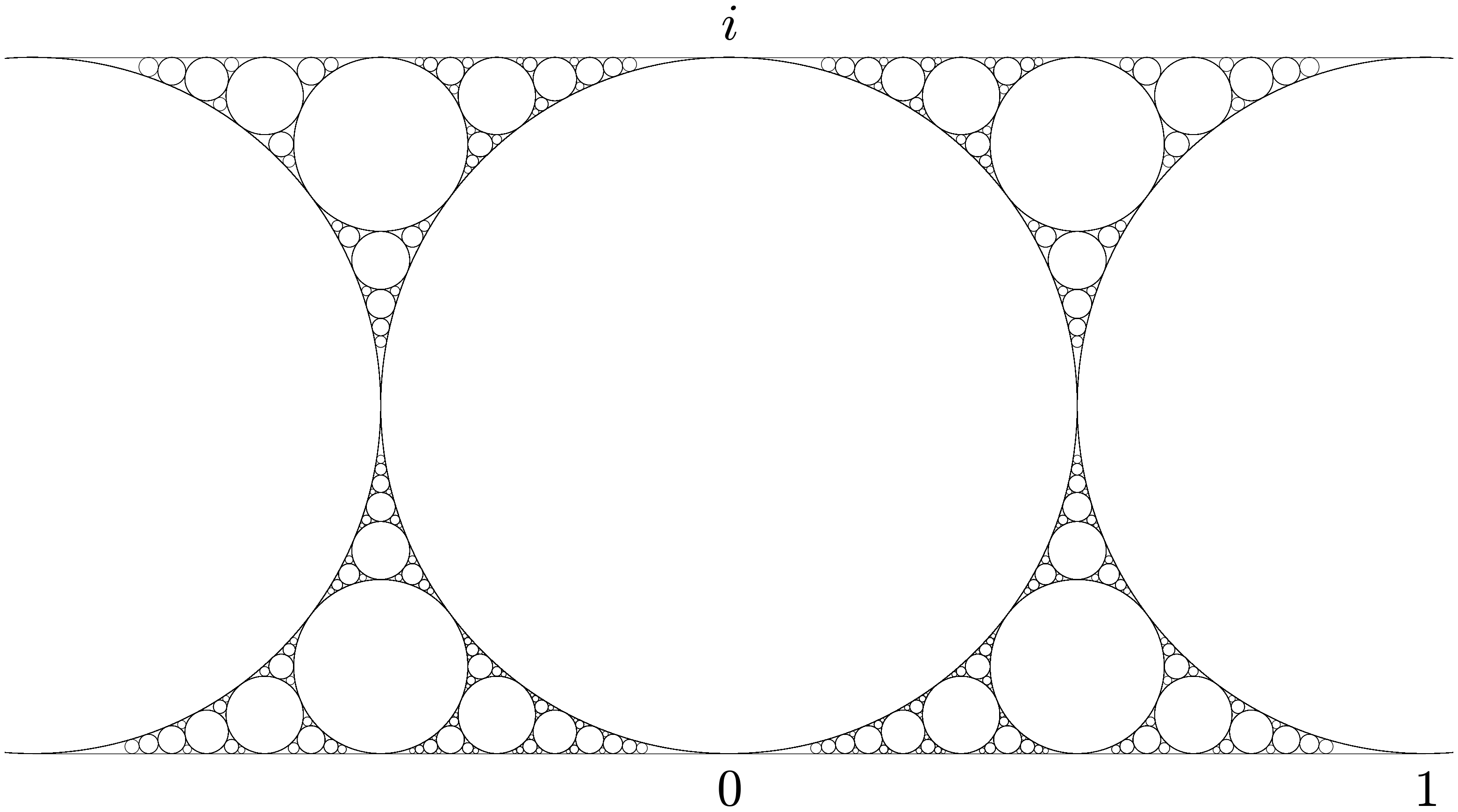}
        \caption{The Apollonian strip packing, obtained as all images of $\RR$ under 
        \usebox{\strip}.}
        \label{fig:strip}
\end{figure}

\newsavebox{\nohalf}
\savebox{\nohalf}{
 $\begin{pmatrix} i & 0 \\ 1 & i \end{pmatrix} 
\left\langle \begin{pmatrix} 1 & 0 \\ 1 & 1 \end{pmatrix}, \begin{pmatrix} 0 & i \\ i & 1 \end{pmatrix}\right\rangle$}

\begin{figure}[ht!]
        \centering
        \includegraphics[width=3.7in]{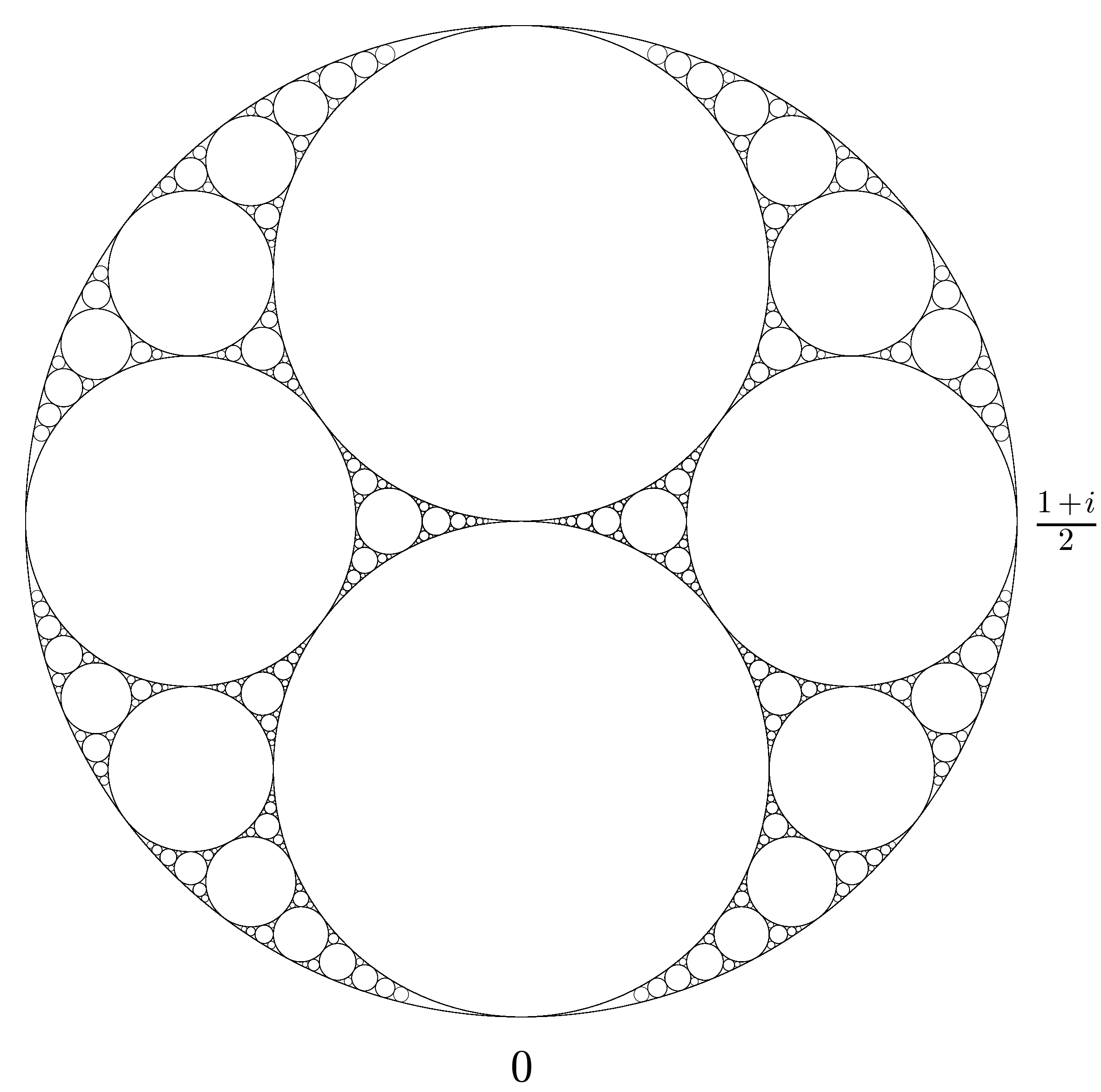}
        \caption{An Apollonian packing, obtained as all images of $\RR$ under the coset \usebox{\nohalf}.}
        \label{fig:1223nohalf}
\end{figure}

An Apollonian circle packing or a Descartes quadruple in $\CC_\infty$ is called \emph{integral} if all of its curvatures are integral, and \emph{strongly integral} if, in addition, the centre of each circle, multiplied by the curvature, is a Gaussian integer.  It is called \emph{primitive} if it is integral and the collection of curvatures has no non-trivial common factor.  As we shall see later, packings obtained in the manner of Figures \ref{fig:strip} and \ref{fig:1223nohalf} are, \emph{once they are dilated by a factor of two in each dimension}, primitive and strongly integral.

This dilation by two will be a persistent housefly in our story -- a visitor who never does more than buzz about irritatingly, but who cannot be induced to leave.  To avoid confusion, we will use the terms \emph{half-curvature} for $1/2$ of the curvature, and \emph{half-primitive} for a packing (or quadruple) which, when dilated by two, is primitive.  The condition of being strongly integral is preserved by dilation or contraction. 

In these example packings, all the circles in question are Gaussian circles.


\subsection{An Apollonian kingdom}
\label{sec:appkingdom}


Now, for a moment we'll forget circles entirely, and turn to generalizing the terminology of Conway and Fung.

\begin{definition}
        Let $R$ be an integral domain with unit group $U$ and field of fractions $F$.  A \emph{lax vector} of $R^2$ is a primitive vector (not an $R$-multiple of any other), considered up to multiplication by elements of $U$.  A \emph{basis} is a pair of lax vectors which, once representatives are chosen, form a basis for $R^2$.  A \emph{superbasis} is a triple of lax vectors any two of which form a basis. 
\end{definition}

The collection of lax vectors is in bijection with $\mathbb{P}^{1}(F)$.  A geometer may prefer the following equivalent definition of bases and superbases:  if we consider $\PP^1(F)$ as $\PP^1(R)$ over $\operatorname{Spec} R$, then bases and superbases are, respectively, pairs and triples of distinct $R$-points of $\PP^1(R)$.

Consider the lax vectors of $\ZZ[i]^2$, which are in bijection with $\PP^1(\QQ(i))$.  A superbasis can be visualized as a triangle (a complete graph on three vertices) whose edges are labelled with the superbasis elements $u$, $v$ and $w$:
\begin{center}
\includegraphics[height=1.8in]{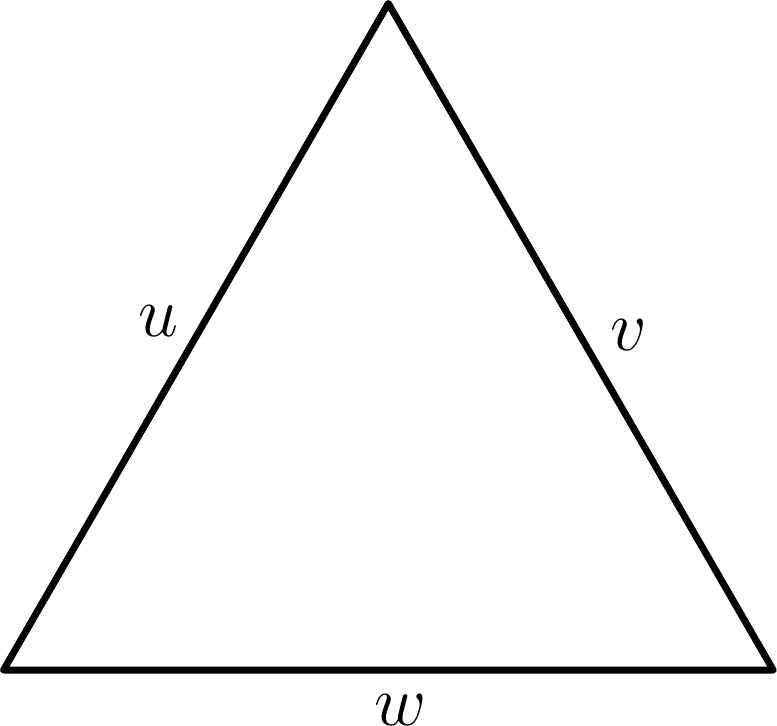}
\end{center}
We will call this a \emph{wall}.

\begin{definition}
An \emph{ultrabasis} is a collection of six lax vectors $\{u,v,w,x,y,z\}$ such that the four subsets $\{u,v,w\}$, $\{u,y,z\}$, $\{v,x,z\}$ and $\{w,x,y\}$ are superbases.
\end{definition}

Whenever the union of four superbases forms an ultrabasis, we can imagine the sides of the corresponding walls identified according to their labels, forming a tetrahedron (a complete graph on four vertices) which consists of the four triangles:
\begin{center}
\includegraphics[height=2in]{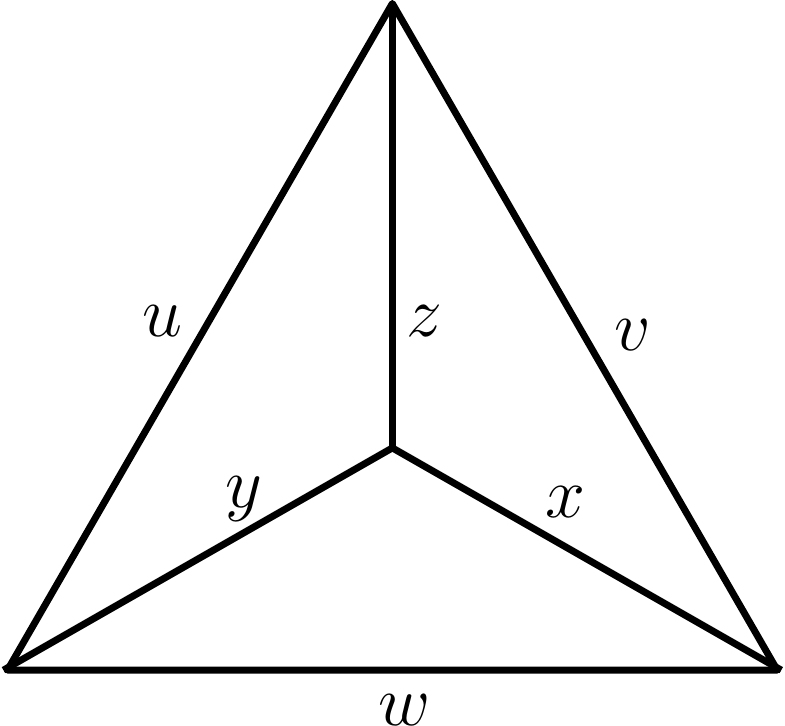}
\end{center}
We will call this a \emph{chamber} (consisting of four walls).  Two ultrabases are considered the same if they form the same labelled graph, up to isomorphism (in other words, an ultrabasis does not have an ordering, or a preferred face, etc.).  As it turns out, if a set of six lax vectors is an ultrabasis, it is an ultrabasis in exactly two ways, having different faces (obtained by taking the dual of the tetrahedron in the sense of Platonic solids, swapping faces for vertices and vice versa).  We consider these distinct. 

If the lax vectors of a superbasis have representatives $u$, $v$ and $w$, then $\epsilon_1 u+\epsilon_2 v+\epsilon_3 w=0$, for some appropriate choice of units $\epsilon_i$.  
Suppose these representatives are chosen with the convention that $u+v+w=0$.
Then it is straightforward to check that there are exactly two ways to form an ultrabasis containing this superbasis, shown in Figure \ref{fig:standard}.
For example, the leftmost triangle is a superbasis satisfying
\[
 u - i( u + i v) + (w + iu) = u +v+w =0.
\]

\begin{figure}[ht!]
        \centering
        \raisebox{-.5\height}{\includegraphics[height=1.9in]{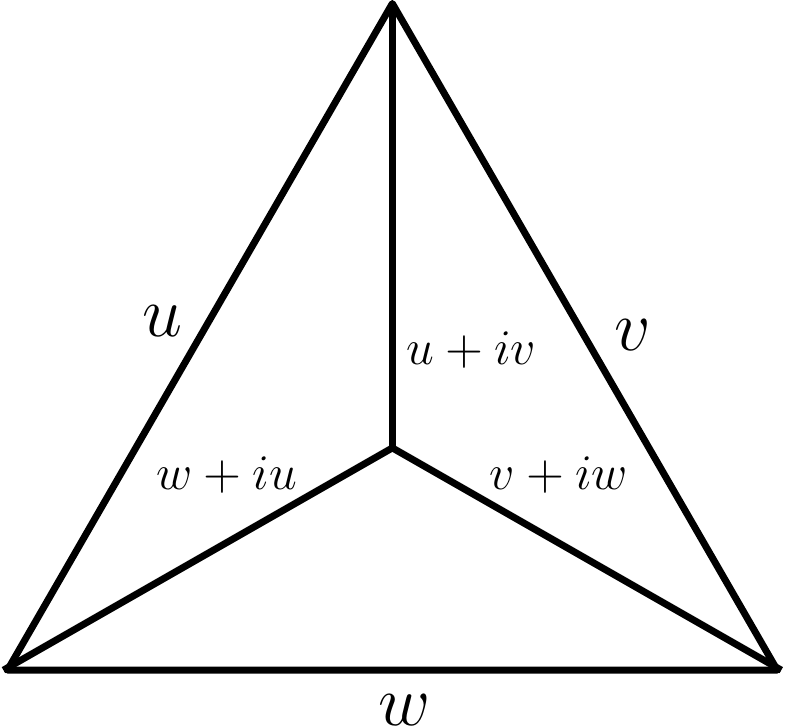}}
\quad \raisebox{-.5\height}{\includegraphics[height=1.9in]{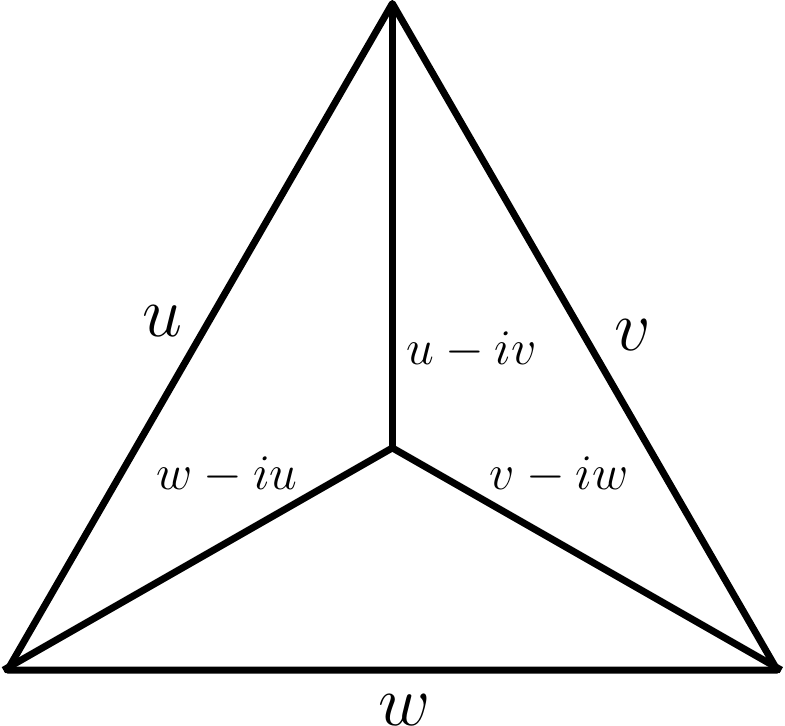}}
        \caption{Left and right standard ultrabases.}
        \label{fig:standard}
\end{figure}

Any ultrabasis is of the form of one of the two \emph{standard ultrabases} just exhibited, for some $u,v,w$; and any pair of chambers sharing a wall is of the form of this pair\footnote{Such representations are not unique.}.  (This is simply a consequence of the fact that any superbasis has representatives satisfying $u+v+w=0$.)  

\begin{definition}
The \emph{Apollonian kingdom} is the graph with labelled edges formed from the collection of all ultrabasis chambers by identifying them along common superbasis walls (matching the lax vector labels on their sides).  
\end{definition}

Each wall appears exactly once, connecting exactly two chambers.


\subsection{A kingdom of palaces}

Any two tangent Gaussian circles have a point of tangency in $\QQ(i)_\infty$, since this point is a double root of a quadratic equation with $\QQ(i)$-coefficients\footnote{It is the fixed point of the map $M_1FM_1^{-1}M_2FM_2^{-1}$, where $F$ is complex conjugation and $M_j$ is the M\"obius map from $\RR$ to the $j$-th circle, which reduces to a quadratic condition.}.  
Because the extended complex plane is a model of $\PP^1(\CC)$, such a point can be associated to lax vector of $\ZZ[i]^2$.

\begin{definition}
        The unique \emph{edge-labelling} for any Apollonian palace is as follows:  label each edge with the lax vector associated to the corresponding point of tangency. 
\end{definition}

\begin{figure}[ht!]
        \centering
        \includegraphics[width=5.7in]{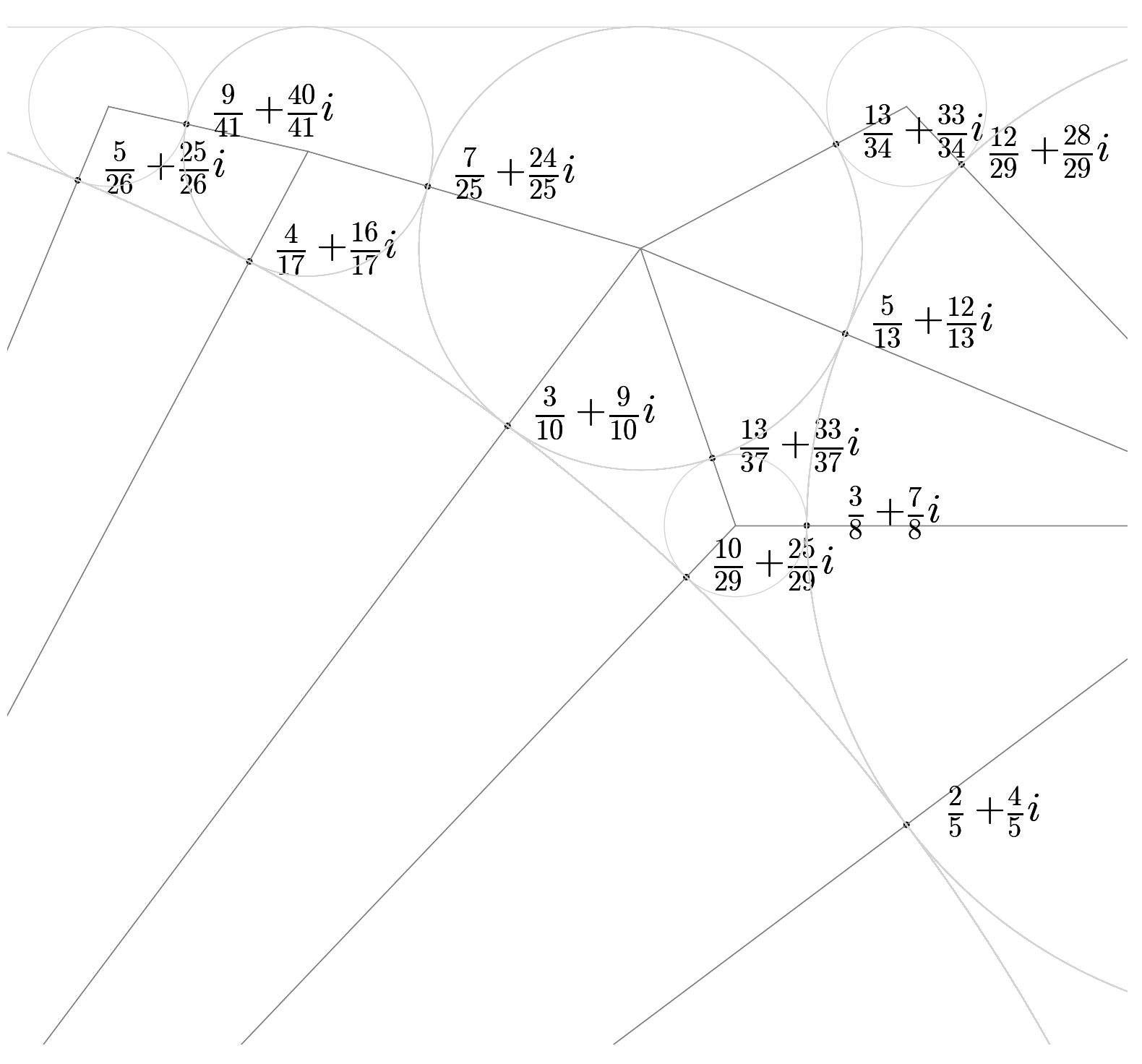}
        \caption{Several circles from the Apollonian strip packing of Figure \ref{fig:strip}.  The Apollonian palace is shown, as are the points of tangency, one of which lies on each edge of the palace.}
        \label{fig:striplax}
\end{figure}

For the strip packing, a few of these are labelled in Figure \ref{fig:striplax}, where one may observe the following remarkable pattern:
        \emph{the lax vector labels on any Apollonian palace are arranged so that, throughout, triangles are superbasis walls, and tetrahedra are ultrabasis chambers.}

For example, in Figure \ref{fig:striplax}, the points of tangency $\frac{5 + 12i}{13} = \frac{3+2i}{3-2i}$, $\frac{3+7i}{8} = \frac{5+2i}{4-4i}$ and $\frac{13+33i}{37} = \frac{3+5i}{6-i}$ lie on three sides of a triangle.  These correspond to primitive vectors forming a superbasis, as illustrated by the identity
\[
        i\begin{pmatrix} 3 + 2i \\ 3 - 2i \end{pmatrix} 
        +\begin{pmatrix} 5 + 2i \\ 4 - 4i \end{pmatrix}
        -\begin{pmatrix} 3+ 5i \\ 6-i\end{pmatrix} = \begin{pmatrix} 0 \\ 0 \end{pmatrix},
\]
and the fact that $2 \times 2$ matrices formed of these columns have unit determinant.

This illustrates the main theorem of this paper:

\begin{theorem}[Main Bijection] 
        \label{thm:mbmain}
The components of the Apollonian kingdom are exactly the edge-labelled Apollonian palaces corresponding to all Gaussian Apollonian circle packings in $\CC_\infty$.  The resulting map from vertices of the kingdom to Gaussian oriented circles in $\CC_\infty$ is a bijection.  
\end{theorem}



The proof involves attaching additional structure to the Apollonian kingdom and its palaces.  In addition to edge labels, in the next section we will attach \emph{lax lattices} to the vertices, and demonstrate how these correspond to Gaussian circles, proving Theorem \ref{thm:mbfirst} (the second statement in the theorem above).  The following Section \ref{sec:herm} discusses Hermitian forms which give the curvatures and centres of Gaussian circles.  In Section \ref{sec:kocik}, we review the criterion in curvatures and centres which distinguishes Descartes quadruples.  This uses a perspective due to Kocik, in which one embeds Descartes quadruples as space-like vectors in spacetime.  In Section \ref{sec:lorentz} we extend this perspective and complete the proof of the main bijection.

\section{Lax lattices and Conway's summer residence}

In Conway's palace, the vertices are the important pieces:  the lax vectors.  These are the raw material fed into the quadratic form to populate the topograph with its values.  In preparation for the proof of the Main Bijection, we will interpret the vertices of the Apollonian kingdom as \emph{lax lattices}, to be defined shortly.  We will also see that, independently, Gaussian circles naturally have an interpretation as lax lattices:  it is this relationship that mediates the Main Bijection.

The definition of a lax lattice will bring to light the fact that our motivating analogy is more than just an analogy:  a certain subgraph of the Apollonian kingdom is a copy of Conway's palace.  In a sense, Conway has a summer residence -- in fact, a great many summer residences -- in the Apollonian kingdom.


\subsection{The lax lattices of an Apollonian palace}
The definition of a lax lattice actually comes about more naturally from the perspective of Gaussian circles, so it is there that we will begin.
Recall that a Gaussian circle is the image of $\RR$ under some transformation 
\[
        M = \begin{pmatrix}a & b \\ c & d \end{pmatrix} \in \PGL_2(\ZZ[i]).
\]
The rational points $\QQ$, under $M$, map to all of the $\QQ(i)$-points on the circle.  Under the correspondence between $\QQ(i)_\infty$ and $\PP^1(\QQ(i))$ (viewed as lax vectors), the points of $\QQ$ are the lax vectors with representatives in $\ZZ^2 \subset \ZZ[i]^2$, and their images under multiplication by $M$ are primitive vectors of the rank-two $\ZZ$-lattice
\[
       x\begin{pmatrix} a \\ c \end{pmatrix} + y\begin{pmatrix} b \\ d \end{pmatrix}, \quad x,y \in \ZZ.
\]

This lattice is not well-defined in terms of $M$:  changing $M$ by a scalar multiple produces a scalar multiple of this lattice.  However, its equivalence class up to scaling is well-defined for $M$.  
Requiring $M$ to have Gaussian integer entries leaves us with two choices, differing by multiplication by $i$, which we will take together as forming an equivalence class.

However, it is still not yet clear that it is well-defined in terms of the circle, since many transformations $M$ may map $\RR$ to the circle.  The reader is invited to verify that the difference between two such transformations is precomposition by a transformation fixing $\RR$.  
If we consider only Gaussian circles, then two transformations $M,M' \in \PGL_2(\ZZ[i])$ mapping $\RR$ to the same circle are related by
\[
        M = M' T, \quad T \in \PGL_2(\ZZ).
\]
(Here, $\PGL_2(\ZZ)$ consists of those matrices of $\PGL_2(\ZZ[i])$ which can be expressed with rational integer entries.)  The multiplication by $T$ changes the matrix, but it preserves the lattice: it is just a change of basis.  So the lattice is well-defined for a Gaussian circle.

Finally, we must consider orientation, which adds a small layer of complication:  we will replace $\PGL_2(\ZZ)$ with $\PSL_2(\ZZ)$.  This is the \emph{special linear group}, which we define as follows for a commutative ring $R$:
\[
        \PSL_2(R) = \left. \left\{ \begin{pmatrix} a & b \\ c & d \end{pmatrix} : ad - bc  = 1; a,b,c,d \in R \right\} \right/ R^* .
\]

The transformations of $\PGL_2(\ZZ[i])$ taking $\RR$ to itself, \emph{and preserving orientation}, are given\footnote{We think of $\PSL_2(\ZZ)$ as a subset of $\PGL_2(\ZZ[i])$ defined as those equivalence classes of matrices which contain an element of determinant $1$ with entries in $\ZZ$.} by $\PSL_2(\ZZ)$.
Thus, the collection of M\"obius transformations of $\PGL_2(\ZZ[i])$ taking $\RR$ to a given Gaussian circle (with orientation) forms a left coset of $\PSL_2(\ZZ)$.


Therefore, to take into account the orientation of the circle, the lattice defined above should be considered an \emph{oriented} $\ZZ$-module.  Orientation is endowed by the orientation of the basis; a unimodular (i.e.\ $\PSL_2(\ZZ)$) change of basis preserves orientation, while a change of basis of determinant $-1$ swaps it.  Such transformations have the same effect on the orientation of circles, so that these two notions of orientation are linked.

The conclusion:  circles are lattices!  \emph{Lax lattices}, to be precise:

\begin{definition}
\label{def:laxvertexlattice}
Considering $\ZZ[i]^2$ as a $\ZZ$-module of rank four, let $L$ be an oriented $\ZZ$-submodule of rank two generated by an oriented $\ZZ[i]$-basis for $\ZZ[i]^2$.  The 
pair $\{ L, iL \}$ is called a \emph{lax lattice}.
\end{definition}

In fact, the number theorist may notice that lax lattices of $\ZZ[i]^2$ fall into classes that are in bijection with classes of fractional ideals in all orders of $\ZZ[i]$.  This interpretation is further pursued in an upcoming paper.

\begin{proposition}
        \label{prop:circlelaxvertexlattice}
        Gaussian circles are in bijection with lax lattices under the association described above.
\end{proposition}

\begin{proof}
        The association describes a well-defined map from circles to lattices.  Since two different circles cannot have equal sets of $\QQ(i)$-points, the map is injective.  To see that it is surjective, take any ordered basis of any lax lattice $L$ (in the sense of taking an ordered basis for one of its representative lattices), and form a matrix $M$ with these column vectors.  Then, it is immediate that the circle $M(\RR)$ has lax lattice $L$.
\end{proof}

Given an Apollonian palace, this gives a \emph{lax lattice} for each vertex.  An example quadruple of Gaussian circles from Figure \ref{fig:1223nohalf} is given with lax lattices in Figure \ref{fig:quadLVL}.


\begin{figure}
        \centering
        \includegraphics[width=4.9in]{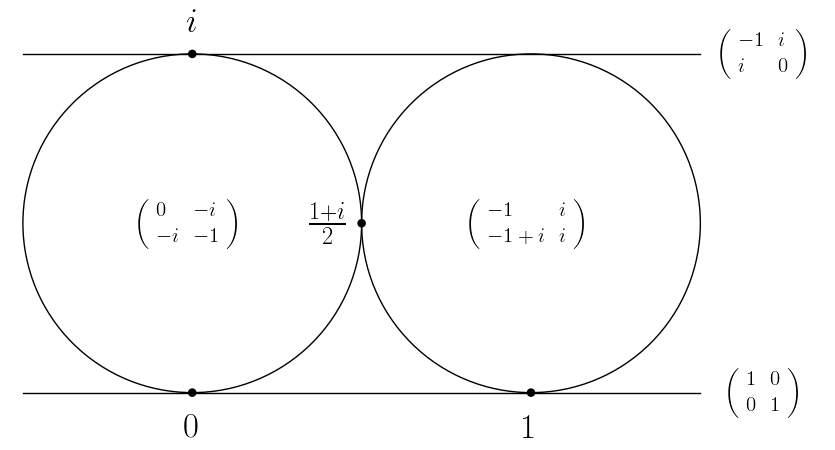}
        \caption{A Descartes quadruple (called the \emph{base quadruple}) from the packing of Figure \ref{fig:strip}, with lax lattices for each circle shown in the form of a pair of column vectors forming a $\ZZ$-basis.}
        \label{fig:quadLVL}
\end{figure}

\begin{figure}
        \centering
        \includegraphics[width=4.9in]{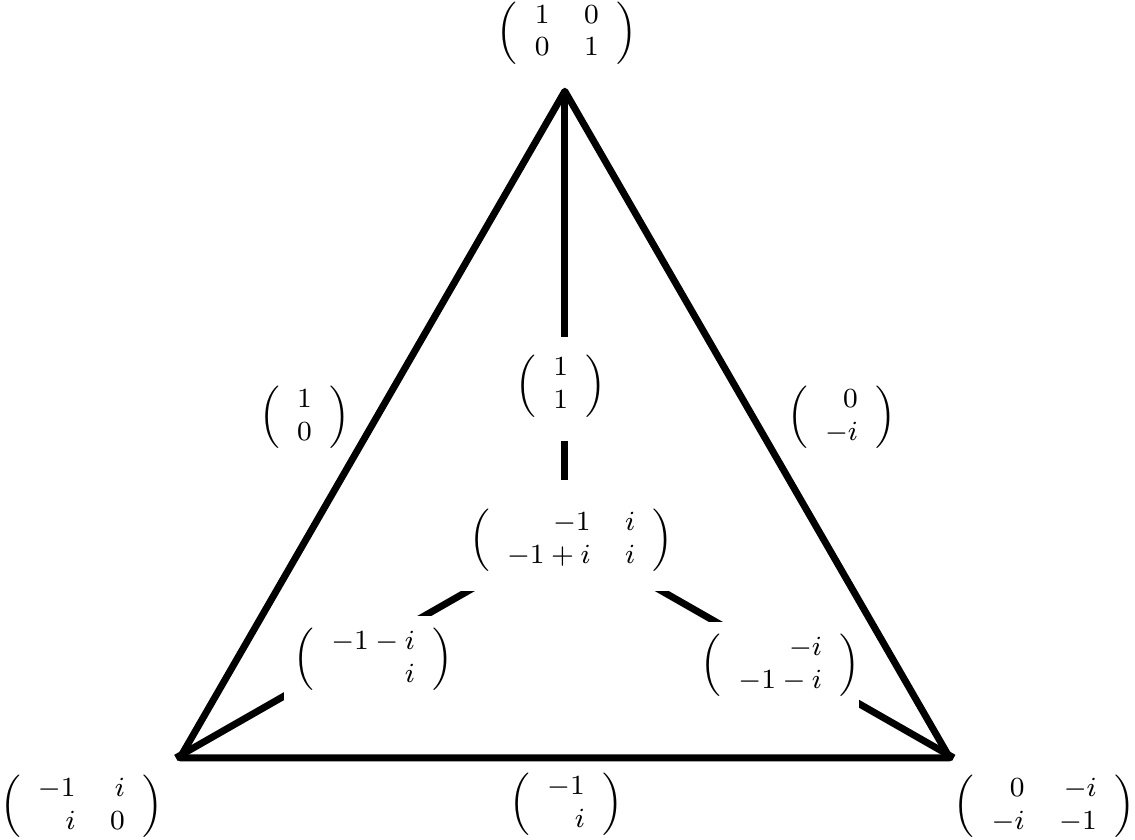}
        \caption{With $u = (1,0)^T$, $v=(0,-i)^T$, $w =(-1,i)^T$, one obtains the ultrabasis chamber associated to the quadruple of Figure \ref{fig:quadLVL}.  This is called the \emph{base chamber}.}
        \label{fig:chamberLVL}
\end{figure}

It is well-defined to speak of a lax vector $\{ v, -v, iv, -iv\}$ being contained in a lax lattice $\{L, iL\}$ (i.e.\ $v \in L$ or $v \in iL$).  Tangent circles contain a unique common point; this is given by a lax vector common to the two lax lattices.  So a lax lattice associated to the circle at one vertex of an Apollonian palace contains all the lax vectors labelling its adjacent edges.  For example, the upper-left-most circle in Figure \ref{fig:striplax} has a lax lattice containing\footnote{Since the corresponding rational points are $\frac{5i}{5+i} = \frac{5}{26}+\frac{25}{26}i$ and $\frac{-4+5i}{4+5i}=\frac{9}{41} + \frac{40}{41}i$, which label the nearby tangency points.} the lax vectors $(5i, 5+i)^T$ and $(-4+5i, 4+5i)^T$.   In fact, the lax lattice for that circle is given by taking these vectors as an ordered basis. 

We now have a labelling for both edges and vertices of an Apollonian palace.  


\subsection{The lax lattices of the Apollonian kingdom}
\label{sec:kingdomvertices}


To prove the Main Bijection, we will decorate the vertices of the Apollonian kingdom by lax lattices without reference to Gaussian circles.  Then the connection between lattices and circles of the previous section completes the bijection between vertices and circles.

The decoration will be straightforward to define:  simply label each vertex of the standard ultrabases of Figure \ref{fig:standard} with an ordered pair of vectors, according to Figure \ref{fig:standardlattices}.  These vectors then serve as a basis for a lax lattice.  Since every chamber of the kingdom can be expressed as a standard ultrabasis, we obtain a labelling of the entire kingdom.

\begin{figure}
        \centering
        \includegraphics[width=2.3in]{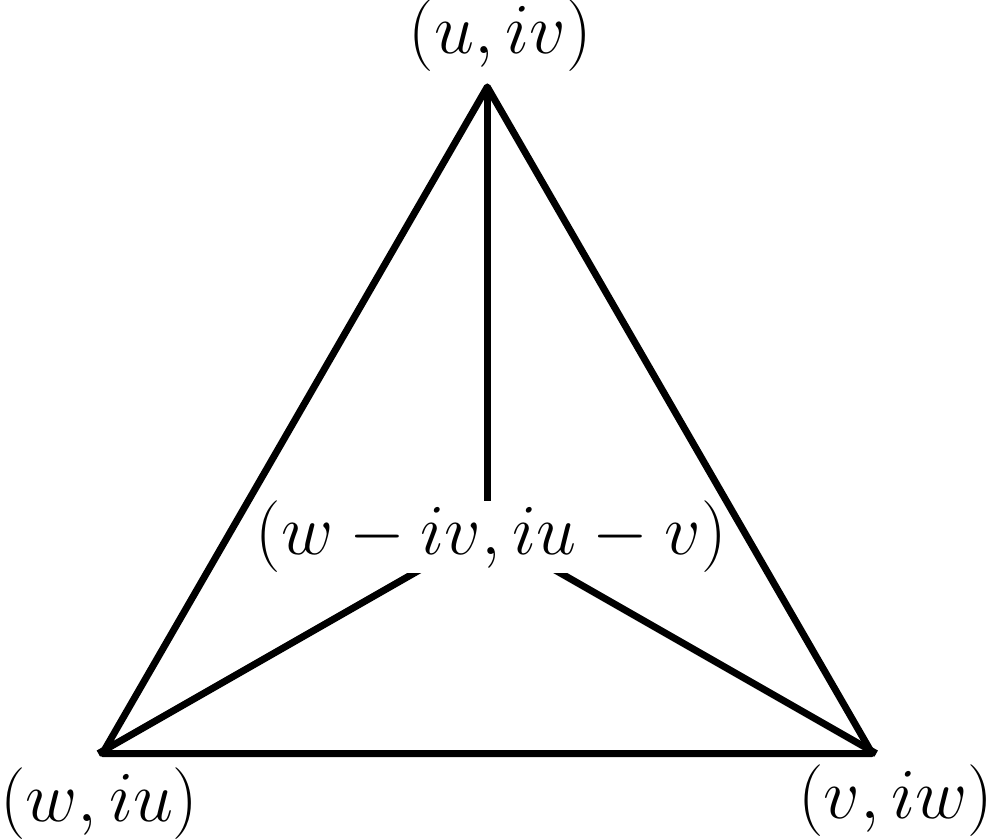}
        \quad \includegraphics[width=2.3in]{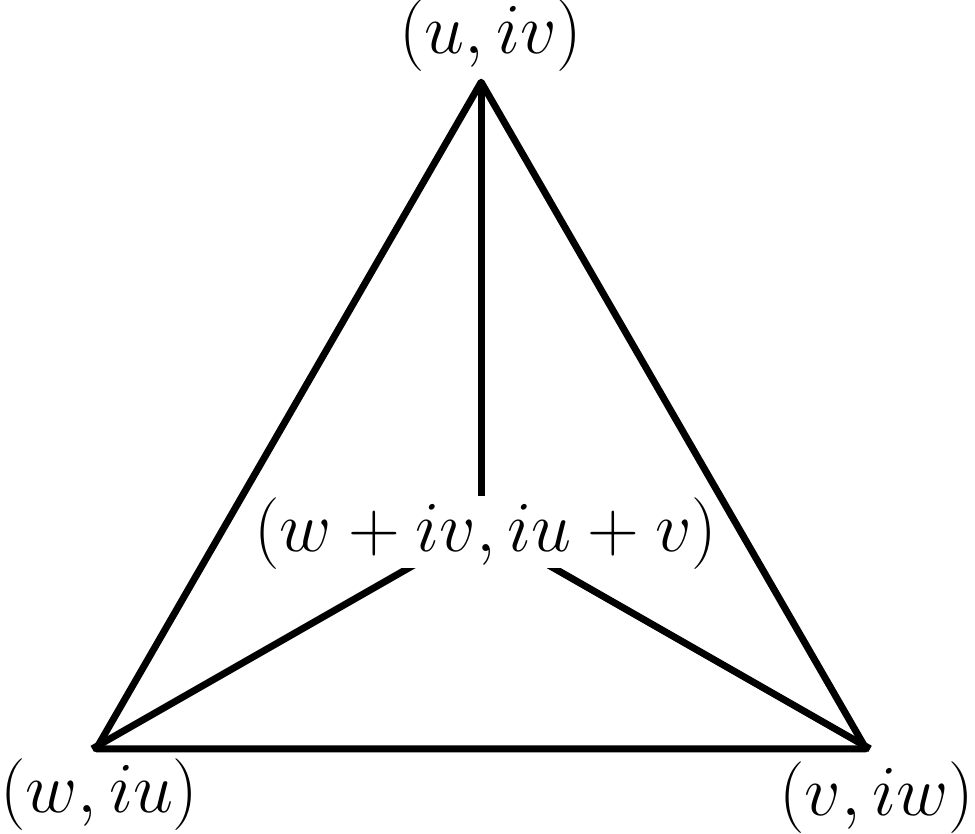}
        \caption{The left and right standard ultrabases of Figure \ref{fig:standard}, labelled with lax lattices.}
        \label{fig:standardlattices}
\end{figure}



\begin{proposition}
        \label{prop:laxwelldefined}
        The lax lattice attached to a vertex in the Apollonian kingdom is well-defined.  It is the unique lax lattice, up to orientation, containing all the lax vectors on edges adjacent to the vertex.
\end{proposition}

\begin{proof}
It is a finite, if slightly tedious, calculation to check that if a given ultrabasis is labelled as a standard ultrabasis in a different way (i.e.\ choosing a different $u$, $v$, $w$), the lax lattice labels of its vertices are unchanged.  It then remains to observe that any two adjacent ultrabases are of the form of the standard pair, and for such a pair, their common vertices agree in labelling.

The lax lattice clearly does contain all the lax vectors on edges adjacent to its vertex.  In fact, every lax vector adjacent to the vertex forms part of a basis for the lax lattice.  Therefore, if another lax lattice contains all these lax vectors, it also contains the first lax lattice.  But two lax lattices cannot be related by proper containment, as their $\ZZ$-bases must be $\ZZ[i]$-bases for $\ZZ[i]^2$, and therefore are related by an invertible change of basis over $\ZZ[i]$, and hence necessarily over $\ZZ$.
\end{proof}

An example chamber is shown in Figure \ref{fig:chamberLVL}.  
In the next section, we will show that the map just defined from vertices to lax lattices is, in fact, a bijection.

\subsection{Conway's summer residence}
\label{sec:prince}

Suppose one were to stand at a chosen vertex of the Apollonian kingdom -- call it the \emph{prince} -- and survey the adjacent vertices, which we will call the \emph{subjects} making up his \emph{court}.

Each subject is connected by an edge to the prince.  That edge is labelled with a lax vector which is an element of the prince's lax lattice.  If we fix a choice of representative lattice (say $L$ instead of $iL$), then this lax vector (an equivalence class of four vectors) contains exactly two vectors which are contained in $L$, i.e.\ $\{ u, -u \}$ for some $u$.  Identifying a subject with the edge connecting it to the prince, we can consider each subject to be labelled with such a `$\ZZ$-lax vector' $\{u, -u\}$ contained in the rank two $\ZZ$-module $L$.  

\begin{definition}
        Choosing a vertex of the Apollonian kingdom, called the \emph{prince}, we define the \emph{court} of the prince to be the subgraph of the Apollonian kingdom induced by the vertices adjacent to the prince.  The vertices are called \emph{subjects}.  Each subject has a \emph{title}, which is a label by a $\ZZ$-lax vector of the prince's lattice, as described in the preceeding paragraph.
\end{definition}


\begin{proposition}
        \label{prop:prince}
        The court, labelled by its titles, has one vertex for each title in $L$, such that any two are connected exactly when they form a basis for $L$.  In other words, it forms a copy of Conway's palace, under any isomorphism of $L$ with $\ZZ^2$.
\end{proposition}

\begin{proof}
Two subjects are connected exactly when the two lax vectors label two sides of a wall touching the prince.  Since two edges of a wall form a basis for the lax lattice, any two connected subjects form such a basis.  Choosing an isomorphism of $L$ with $\ZZ^2$, the graph of subjects is a subgraph of Conway's palace.

But the graph of subjects also contains every wall (triangle) opposite the prince in any chamber adjacent to the prince.  Hence its dual (the graph consisting of vertices for each such wall, connected when walls touch) is a valence three graph, and a subgraph of Conway's topograph.  Since Conway's topograph is connected and of valence three \cite{\Conway}, it is all of Conway's topograph. 
\end{proof}

\subsection{Proof of the Main Bijection, First Part}

\begin{proposition}
        \label{prop:vertexlattices}
        Each lax lattice appears at exactly one vertex of the Apollonian kingdom.
\end{proposition}

\begin{proof}
        The lax lattice generated by $u, v$ appears on one vertex of the superbasis wall $u, -iv, -u+iv$.  Running over all ordered bases $u,v$ of $\ZZ[i]^2$, we find all possible lax lattices.
        
        Conversely, if $(u,v)$ is an ordered basis of a lax lattice $L$ associated to a prince, then by Proposition \ref{prop:prince}, the prince touches the superbasis wall $u, -iv, -u+iv$ between the sides $u$ and $v$.  Since superbasis walls appear only once in the Apollonian kingdom, this identifies the prince uniquely:  each lax lattice appears only once.\footnote{Note that for each lax lattice there is another having the same representative lattices, oriented oppositely.  It will have basis $(v, u)$ or $(u, -v)$ and touch superbasis wall $u, iv, -u-iv$.}
\end{proof}

Propositions \ref{prop:circlelaxvertexlattice} and \ref{prop:vertexlattices} together give us the first part of our central result.

\begin{theorem}[Main Bijection, First Part]
        \label{thm:mbfirst}
        The vertices of the Apollonian kingdom are in bijection with the collection of Gaussian circles. 
\end{theorem}

What remains is to show that their arrangement reflects the arrangement of Gaussian circles into Descartes quadruples and Apollonian circle packings.  To be precise, we will show that adjacency in the kingdom exactly corresponds to tangency \emph{within a Descartes quadruple} (in other words, not all tangent circles are adjacent vertices).  To reach this destination, however, we must first take a detour to study Hermitian forms.

\section{Hermitian Forms}
\label{sec:herm}

Conway defined the topograph to study the values of quadratic forms.  He labelled each region with a lax vector, and then fed those vectors into a quadratic form, to obtain a labelling of the topograph regions with the values of the form.

In our case, the corresponding objects are the vertices of the Apollonian kingdom, corresponding to circles and to lax lattices.  Instead of a quadratic form, we will evaluate the imaginary part of a Hermitian form at each of these.  These values satisfy a `linear Descartes-like relation' like \eqref{eqn:desclinear} for curvatures.  In fact, the values obtained, for a special choice of form, \emph{are} the curvatures of the circles.

\subsection{Labelling the Apollonian kingdom with a Hermitian form}
\label{sec:hermitian}

A \emph{Hermitian form} is a pairing
\[
h: \CC^n \times \CC^n \rightarrow \CC
\]
satisfying for any $u, v, w \in \CC^n$ and $a \in \CC$,
\begin{enumerate}
\item $h(au + v, w) = ah(u,w) + h(v,w)$
\item $h(u, av +w) = \overline{a}h(u,v) + h(u,w)$
\item $h(u,v) = \overline{h(v,u)}$
\end{enumerate}
(Note that the first and third entail the second of these conditions.)  In our case, $n=2$.  We will study $H$, the imaginary part of $h$ (i.e.\ the coefficient of $i$).  Then $H$ has the following important property:
\begin{equation}
        \label{eqn:gamma} H(au+bv, cu+dv) = (ad-bc)H(u,v) \mbox{ for }a, b, c, d \in \RR. \tag{H1}
\end{equation}
(Note: since $ih(u,v) = h(u, -iv)$, we may have equivalently chosen to study the collection of real parts of $h$; in fact, $h$ is determined by $H$.)

From now on, we will assume $h$ is \emph{integer valued} on $\ZZ[i]^2 \times \ZZ[i]^2$, i.e.\ it takes values in $\ZZ[i]$.  

We can evaluate $H$ at a vertex of the Apollonian kingdom by applying $H$ to a basis of the lax lattice.  Properties \eqref{eqn:gamma} and conjugate linearity together guarantee that the $H$-value is well-defined.  


%

\begin{lemma}\label{lemma:quantitiesintegers}
If $u + v + w = 0$, then the following quantity is real:
\[
        h(w-iv,iu-v) + h(w+iv,iu + v) - h(iv,v) - h(iw,w) - h(iu,u)
\]
\end{lemma}
\begin{proof}
The proof is by direct calculation.  It is convenient to note that
for a Hermitian form, the following quantities are real:
\[
        h(u,u),\quad
h(u,v) + h(v,u), \quad
        i( h(u,v) - h(v,u) ).
\]
Note also that $h(ix,iy) = h(x,y)$ and that $h(ix, y) = -h(x, iy)$ for all $x, y$.
\end{proof}

As a consequence, we discover that the $H$-labels on the pair of standard ultrabase tetrahedra have the following relationship:
        \begin{equation}
                \label{eqn:Hdesc}
                H(w-iv,iu-v) + H(w+iv,iu+v) = 2( H(u,iv) + H(v,iw) + H(w,iu) ) \tag{H2}
        \end{equation}

        This proves Theorem \ref{thm:herm}:  property \eqref{eqn:Hdesc} can be considered the linear Descartes rule for Hermitian forms.
%
%
%
%
By Proposition \ref{prop:vertexlattices}, the collection of Hermitian values
\[
        \{ H(u,v) : u,v \mbox{ form a $\ZZ[i]$-basis for $\ZZ[i]^2$} \}
\]
is now displayed on the vertices of the Apollonian kingdom.

\subsection{The curvature and centre of a circle in $\PP^1(\CC)$}
\label{sec:curvature}

We will now see that the fundamental parameters of a circle $M(\RR)$ -- its curvature and centre -- are determined by Hermitian forms on the columns of $M$.  Bythe \emph{curvature} of an oriented circle, we mean a quantity that is positive if the circle is oriented clockwise, negative if oriented counterclockwise, and is equal to the usual curvature (inverse of the radius) in absolute value.  The curvature of a line is zero.

We will also define the \emph{co-curvature} of an oriented circle $C$ as the oriented curvature of the circle $C'$ obtained by inversion in the unit circle (computed by $z \mapsto 1/\overline{z}$).  The \emph{half-co-curvature} is half of the co-curvature.

\begin{proposition}
\label{prop:curvature}
Consider an oriented Gaussian circle expressed as the image of $\RR$ (oriented to the right) under a transformation of the form
\[
M =
      \begin{pmatrix} \alpha & \gamma \\ \beta & \delta \end{pmatrix}, \quad \alpha, \beta, \gamma, \delta \in \ZZ[i], \;\; \alpha\delta - \beta\gamma \in \ZZ[i]^*.
\]
The curvature of the circle is an even integer given by
\[
        2 \operatorname{Im}\left( \beta \overline{\delta} \right),
\]
the co-curvature of the circle is an even integer given by
\[
        2 \operatorname{Im}\left( \alpha \overline{\gamma} \right),
\]
and the product of the curvature and centre of the circle (called the \emph{curvature-centre}) is a Gaussian integer given by
\[
        i(\gamma \overline{\beta} - \alpha \overline{\delta}) =
        \operatorname{Im}\left( \beta \overline{\gamma}  +\alpha\overline{\delta}  \right) 
        + i
        \operatorname{Im}\left(  i\beta \overline{\gamma} - i \alpha\overline{\delta}  \right).
\]
Furthermore, this value has even integer real part and odd integer imaginary part whenever $M \in \PSL_2(\ZZ[i])$; it has odd integer real part and even integer imaginary part whenever $M \in \PGL_2(\ZZ[i]) \setminus \PSL_2(\ZZ[i])$.

Finally, an oriented circle of curvature $b$, co-curvature $b'$ and curvature-centre $z$ satisfies the relation
\begin{equation}
        \label{eqn:cocurv}
        b'b = z\overline{z} - 1.
\end{equation}

\end{proposition}

\begin{proof}[Proof of Proposition \ref{prop:curvature}]
        Define the following functions on matrices $M \in \PGL_2(\CC)$:
        \begin{align*}
                H_1 (M) &= 2 \operatorname{Im}( \beta\overline{\delta} ) / |\det(M)|^{1/2}, \\
                H_2 (M) &= 2 \operatorname{Im}( \alpha\overline{\gamma} ) / |\det(M)|^{1/2}, \\
                H_3 (M) &= \operatorname{Im}\left(  \beta \overline{\gamma} + \alpha\overline{\delta}   \right) / |\det(M)|^{1/2}, \\
                H_4 (M) &= \operatorname{Im}\left( i\beta \overline{\gamma}  - i\alpha\overline{\delta}  \right) / |\det(M)|^{1/2}. 
        \end{align*}
        (Exponent $1/2$ denotes the positive square root.)  We will show that the image of the real line under an $M \in \PGL_2(\CC)$ (interpreted as a M\"obius transformation) has curvature $H_1(M)$, co-curvature $H_2(M)$ and curvature-centre $H_3(M) + iH_4(M)$.

The key observation is that each $H_i$ is invariant under replacing such an $M$ by any other transformation $M'=MP$ for some $P \in \PSL_2(\RR)$.  To see this, we choose representative matrices of unit determinant, so that each $H_i$ is the imaginary part of a Hermitian form on the column vectors of the matrix $M$.   Then property \eqref{eqn:gamma} of the imaginary part of a Hermitian form in Section \ref{sec:hermitian} suffices.

Since $M, M' \in \PGL_2(\CC)$ map $\RR$ to the same oriented circle if and only if $M' = MP$ for some $P \in \PSL_2(\RR)$, it suffices to prove the statement for any $M'$ taking $\RR$ to the oriented circle in question.

Every clockwise oriented circle can be expressed as an image of $\RR$ under the transformation
\[
        M' = \begin{pmatrix} \frac{k+b}{\sqrt{2k}} & i\frac{k-b}{\sqrt{2k}} \\ \frac{1}{\sqrt{2k}} & \frac{-i}{\sqrt{2k}} \end{pmatrix}, \quad b \in \CC, k \in \RR^{>0},
\]
since this is a composition of the transformation $\begin{pmatrix} 1/\sqrt{2} & i/\sqrt{2} \\ 1/\sqrt{2} & -i/\sqrt{2} \end{pmatrix}$ taking $\RR$ (oriented to the right) to the clockwise unit circle, with a dilation $\begin{pmatrix} \sqrt{k} & 0 \\ 0 & 1/\sqrt{k} \end{pmatrix}$ and a translation $\begin{pmatrix} 1 & b \\ 0 & 1 \end{pmatrix}$.   The image of the real axis under this $M'$ is the clockwise circle with centre $b$ and radius $k$, and $M'$ satisfies $H_1(M') = 1/k$, and $H_3(M') + iH_4(M') = b/k$.  

Inversion in the unit circle is accomplished by the map $z \mapsto \frac{1}{\overline{z}}$.  Conjugation preserves radii, but reverses orientation. Therefore the map $z \mapsto \frac{1}{z}$ takes a circle with co-curvature $b'$ to a circle with curvature $b'$.  The curvature of the image of $M'$ under the M\"obius transformation $z \mapsto \frac{1}{z}$ is $H_2(M')$.

If we change the orientation of the circle, replacing $M$ with $M \begin{pmatrix} -1 & 0 \\ 0 & 1 \end{pmatrix}$, then the values of the $H_i$ change sign.
%

Finally, return to the restrictions that $\det(M) \in \ZZ[i]^*$, $\alpha, \beta, \gamma, \delta \in \ZZ[i]$.  We have
\begin{equation}
        \label{eqn:detbar}
        \alpha \delta - \gamma\beta - (\alpha\overline{\delta} - \gamma \overline{\beta}) = \alpha( \delta - \overline{\delta}) - \gamma ( \beta - \overline{\beta}) \equiv 0 \pmod 2
\end{equation}
Since $\det(M) = \alpha \delta - \gamma\beta$ is $\pm 1$ or $\pm i$, this means the real and imaginary parts of the curvature-centre $i(\gamma \overline{\beta} - \alpha \overline{\delta})$ consist of one even, and one odd, integer.  In particular, if $M$ is in $\PSL_2(\ZZ[i])$, then $\det(M) = \pm 1$, so the real part is even, and the imaginary part is odd.  If $M$ is in the nontrivial coset of $\PSL_2(\ZZ[i])$ in $\PGL_2(\ZZ[i])$, then the opposite is true.

To verify relation \eqref{eqn:cocurv}, note that $1=N(\alpha\delta-\beta\gamma)$.

\end{proof}

\section{Packing circles in spacetime}
\label{sec:kocik}

In order to finish the proof of the Main Bijection, we will need to make use of a criterion specifying when four Gaussian circles are in Descartes configuration.  Such a criterion, in the form of an extended Descartes relation involving not just curvatures, but also curvature-centres, was given in \cite{\ntone}.  Kocik gave the following particularly compelling proof in \cite{kocikminkowski}, based on observations of Pedoe.  This section reviews Kocik's work. 



Let $\MM$ be the vector space $\RR^4$ with an inner product $\langle, \rangle$ given by
\[
        \mathbf{M}=  
        \begin{pmatrix}
                0 & -\frac{1}{2} & 0 & 0 \\
      -\frac{1}{2} & 0 & 0 & 0 \\
                0 & 0 & 1 & 0 \\
                0 & 0 & 0 & 1 
        \end{pmatrix}.
\]
This is called \emph{Minkowski spacetime}.  We are using the standard isotropic basis, instead of the standard orthonormal basis that is usually introduced with special relativity.  A vector represents a position in three dimensions of space and one of time; in this basis time is measured along the vector $\mathbf{t} = (1,1,0,0)^T$.  A vector points forward in time if its projection onto $\mathbf{t}$ points in the same direction as $\mathbf{t}$.  Vectors of norm $0$ are called \emph{light-like} and form the \emph{light cone}, dividing space into two regions: \emph{time-like} vectors of negative norm; and \emph{space-like} vectors of positive norm.


The fundamental observation is that oriented circles in $\CC_\infty$ map to the unit hyperoloid $\mathbb{H}$ (consisting of vectors of norm $1$ in Minkowski space) in such a way that their inner product describes the manner in which they intersect.  This map, $\pi$, is called the \emph{Pedoe map}.  The image of a circle $C$ of curvature-centre $z$, curvature $b$ and co-curvature $b'$ is
\[
        \pi(C) = 
        \left( b, b', \Re(z), \Im(z) \right)^T,
\]
which has norm $1$ by \eqref{eqn:cocurv} of Proposition \ref{prop:curvature}.  This is a variation on the \emph{augmented curvature-center coordinates} defined in \cite[Definition 3.2]{\beyond}.  For a circle of the form $M(\RR)$ for some $M \in \PGL_2(\ZZ[i])$, Proposition \ref{prop:curvature} gives an alternate form,
\[
        \pi(C) = ( H_1(M), H_2(M), H_3(M), H_4(M) )^T.
\]

Straight lines map into the plane $b=0$, and $\pi(C)$ rotates from this plane toward the light cone as curvatures increase or decrease to $\pm \infty$.  Note that \emph{being positively oriented as a circle is the same as pointing forward in time}.  If we orient $C$ oppositely, calling the result $-C$, then $\pi(-C) =-\pi(C)$.  

\begin{figure}
        \centering
        \includegraphics[width=1.9in]{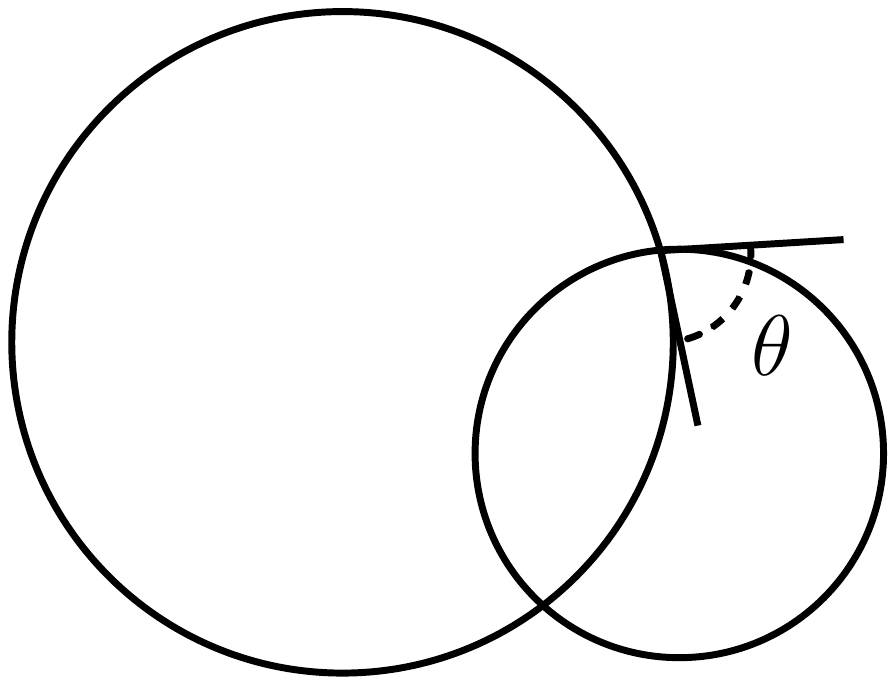}
        \caption{The angle between two non-disjoint circles.}
        \label{fig:anglecircles}
\end{figure}

Kocik credits the following fundamental observation to Pedoe \cite{MR1017034}, albeit not originally given in the language of space-time:

\begin{proposition}[Proposition 2.4 of \cite{kocikminkowski}]
        Let $v_i = \pi(C_i)$ for two circles $C_1, C_2$ which are not disjoint.  Then $\langle v_1, v_2 \rangle = \cos \theta$, where $\theta$ is the angle between the two circles as in Figure \ref{fig:anglecircles}.  In particular,
                \begin{enumerate}
                        \item $\langle v_1, v_2 \rangle = -1$ if and only if the circles are tangent externally \includegraphics[width=0.5in]{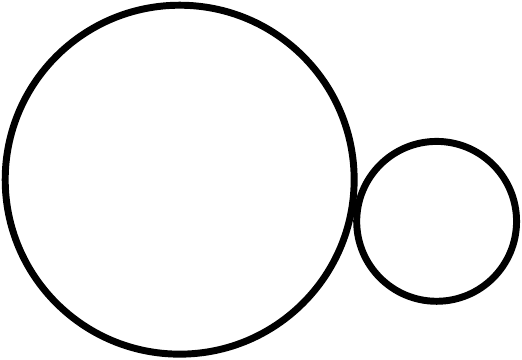}
                        \item $\langle v_1, v_2 \rangle = 1$ if and only if the circles are tangent internally \includegraphics[width=0.38in]{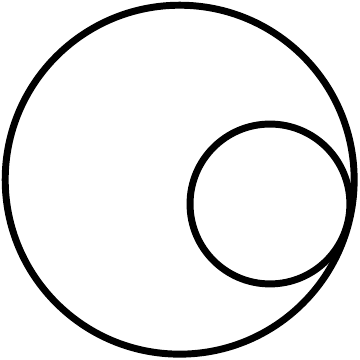}
                        \item $\langle v_1, v_2 \rangle = 0$ if and only if the circles are mutually orthogonal \includegraphics[width=0.47in]{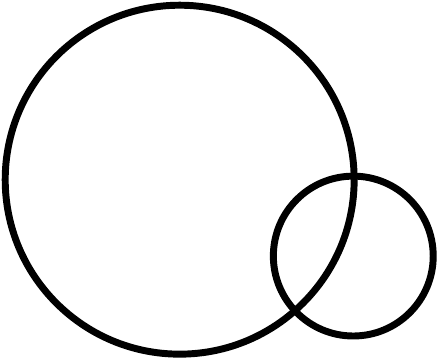}
                \end{enumerate}
\end{proposition}

(We have used a different sign convention than in \cite{kocikminkowski}.)   The following statement now holds simply by definition.

\begin{theorem}[Theorem 3.1 of \cite{kocikminkowski}]
        \label{thm:gtone}
        Let $C_i$ be a quadruple of circles.  Define $\mathbf{R}$ to be the matrix with entries $R_{ij} = \langle \pi(C_i), \pi(C_j) \rangle$.  Define $\mathbf{A}$ to be the matrix whose columns are $\pi(C_i)$.  Then
        \begin{equation}
                \label{eqn:kocik1}
        \mathbf{R} = \mathbf{A}^T \mathbf{M} \mathbf{A}.
        \end{equation}
\end{theorem}

The dramatic title of \emph{Theorem} is warranted, however, because it is a generalisation of the criterion of \cite[Theorem 3.3, $n=2$]{\beyond} for Descartes quadruples.  It tells us that four circles form an \emph{oriented Descartes quadruple} if and only if the equation holds with
\[
       \mathbf{R} = \begin{pmatrix}
                1 & -1 & -1 & -1 \\
                 -1 & 1 & -1 & -1 \\
                 -1 & -1 & 1 & -1 \\
                 -1 & -1 & -1 & 1 \end{pmatrix}.
\]
The term \emph{oriented Descartes quadruple} is defined in \cite{\beyond} to include not just the Descartes quadruples we have discussed thus far (called \emph{positively oriented}), but also their images under a map reversing circle orientation (called \emph{negatively oriented}).  A negatively oriented Descartes quadruple will no longer have disjoint interiors.  We will continue to use the term \emph{Descartes quadruple} to mean only positively oriented Descartes quadruples.

In this case, the matrices are all invertible, so that we may write instead 
       \begin{equation}
                \label{eqn:kocik}
                \mathbf{A} \mathbf{R}^{-1} \mathbf{A}^T = \mathbf{M}^{-1}.
        \end{equation}
        After a change of coordinates, this is the original form due to Lagarias, Mallows and Wilks in \cite[Theorem 3.3]{\beyond}.  Considering the top-left entry of this matrix equation gives the usual Descartes relation \eqref{eqn:descquad}.  Other entries give similar relations for the curvature-centres of a quadruple etc., some of which have their own monikers; see \cite{\beyond}.  

The power of Kocik's proof is that other choices of $\mathbf{R}$ reflect other arrangements of circles, an interesting journey we will not undertake here. 

\section{Lorentz transformations and the proof}
\label{sec:lorentz}

Look up to the heavens surrounding the earth.  In spacetime, the celestial sphere observed from the origin consists of the backward light-like vectors.  These form a three-dimensional cone in spacetime, but as a collection of rays (projectivizing), one obtains a sphere.  The Lorentz transformations are automorphisms of this sphere, and in fact, they are conformal mappings.  For this reason, an observer at relativistic speeds will see the night sky distorted by a M\"obius transformation.  This loosely explains why there exists an isomorphism described by Penrose as ``the first step of a powerful correspondence between the spacetime geometry of relativity and the holomorphic geometry of complex spaces,'' \cite[Chapter 2.6]{KevinBrown}
\[
\phi: \PGL_2(\CC) \rightarrow \SOR 
\]
of $\PGL_2(\CC)$ taking values in the proper orthochronous Lorentz group of matrices, denoted $\SOR$.  This group is defined as all matrices $N$ of determinant $1$ preserving the Minkowski norm and the orientation of time, i.e. $N$ such that
\[
        N^T \mathbf{M} N = \mathbf{M},
\]
and $N(\mathbf{t})$ points forward in time.

For more on this isomorphism, called the \emph{spinor map}, see \cite[Chapter 1.2]{MR776784}.  It restricts to an isomorphism
\[
       \phi: \PGL_2(\ZZ[i]) \rightarrow \SO.
\]
To give this representation explicitly (and see why the last statement is true), we write a circle $[b, b', r, m]$ as a Hermitian matrix
\[
      H =  \begin{pmatrix}
                b' & r+mi \\
             r-mi & b
        \end{pmatrix}
\]
and then $M \in \PGL_2(\CC)$ acts via
\[
        H \mapsto M H\overline{M}^T.
\]
\begin{lemma}
        The M\"obius action of $\PGL_2(\CC)$ on circles in $\CC_\infty$ agrees with the left action of $\SOR$ on $\mathbb{H}$.  
\end{lemma}

\begin{proof}
        Using the functions $H_i$ of Section \ref{sec:curvature}, this is a straightforward computation.
\end{proof}

We have seen that $\PGL_2(\CC)$ preserves tangencies, so that it maps Descartes quadruples to Descartes quadruples.  Alternatively, we can see this from the perspective of $\SO$, because Theorem \ref{thm:gtone} implies that
\[
        \mathbf{A}^T N^T \mathbf{M} N \mathbf{A} =
        \mathbf{A}^T  \mathbf{M} \mathbf{A} = \mathbf{R}.
        \]

\begin{lemma}
        \label{lemma:mod2}
        If $z_1, z_2 \in \ZZ[i]$ are curvature-centres of two Gaussian circles which are tangent, then $z_1 \equiv z_2 \not\equiv 0 \pmod{1+i}$.
\end{lemma}

\begin{proof}
        Write $C_1$, $C_2$ for the two circles, and 
        \[
                \pi(C_j) = ( b_j, b_j', r_j, m_j)^T.
        \]
        Then the curvature-centres are $z_j = r_j + m_j i$.  For Gaussian circles, the entries are all integers.  We have 
        \[
                1 = \langle \pi(C_1), \pi(C_2) \rangle = 
                \frac{1}{2}b_1b_2' + \frac{1}{2}b_1'b_2 + r_1r_2 + m_1m_2.
        \]
        By Proposition \ref{prop:curvature}, the $b_j$ and $b_j'$ are even, 
        and exactly one of $r_j$ and $m_j$ is even.  Examining the equation modulo $2$, we discover 
        $z_1 \equiv z_2 \not\equiv 0 \pmod{1+i}$.
\end{proof}

Full rank matrices $\mathbf{A}$ satisfying \eqref{eqn:kocik} are of determinant $\pm 8$.  If they are of positive determinant, the corresponding ordered oriented\footnote{Such a matrix may represent a positively or negatively oriented Descartes quadruple.} Descartes quadruple (with circles given in order by the columns in the matrix) is called \emph{positively ordered}; otherwise it is \emph{negatively ordered}.  Any Descartes quadruple can be ordered positively in twelve ways and negatively in twelve ways.

        \begin{proposition}
                \label{prop:so}
                Any two positively ordered Descartes quadruples of Gaussian circles are related by a unique element of $\SO$. 
        \end{proposition}

        \begin{proof}
                The matrix for a Descartes quadruple is invertible over $\QQ$, by \eqref{eqn:kocik}.  Let $A$ and $B$ be two such matrices.  Let $N = AB^{-1}$; this is the unique element such that $NB = A$.  To demonstrate that $N \in \operatorname{O}^+_{1,3}(\ZZ)$, first remark that
                \[
                        N^T \mathbf{M} N = (B^{-1})^T A^T \mathbf{M} A B = M
                \]
where the last equality follows from the assumption that $A$ and $B$ are Descartes quadruples, so that
                \[
                        B^T M B = R = A^T M A.
                \]

                Secondly, $N$ preserves the orientation of time if it preserves the orientation of circles.  Both $A$ and $B$ are positively oriented Descartes quadruples, hence they contain at most one negatively oriented circle (since two negatively oriented circles cannot have disjoint interiors).  Therefore, $N$ cannot be orientation-reversing.

                Finally, we wish to show that $N$ has entries in $\ZZ$.  By Lemma \ref{lemma:mod2}, the matrix of any Descartes quaduple is a product of one of the matrices
                \[
                        T_1 = \begin{pmatrix}
                                2 & 0 & 0 & 0 \\
                                0 & 2 & 0 & 0 \\
                                0 & 0 & 2 & 0 \\
                                0 & 0 & 0 & 1 
                        \end{pmatrix},
                        \quad \mbox{or} \quad
                        T_2 = \begin{pmatrix}
                                2 & 0 & 0 & 0 \\
                                0 & 2 & 0 & 0 \\
                                0 & 0 & 1 & 0 \\
                                0 & 0 & 0 & 2 
                        \end{pmatrix}
                \]
                and an integer matrix of determinant $\pm 1$.  Furthermore, with reference to Proposition \ref{prop:curvature} and Lemma \ref{lemma:mod2}, we find that the circles in any quadruple are all images of matrices of $\PSL_2(\ZZ[i])$ or all images of its nontrivial coset in $\PGL_2(\ZZ[i])$.  In other words, there are two types of quadruples, one expressible using $T_1$ and one expressible using $T_2$.  One can map from a quadruple of one type to the other by $z \mapsto iz$, whose image in $\operatorname{O}^+_{1,3}(\ZZ)$ is
                \[
                        \begin{pmatrix}
                                1 & 0 & 0 & 0 \\
                                0 & 1 & 0 & 0 \\
                                0 & 0 & 0 & 1 \\
                                0 & 0 & -1 & 0 \\
                        \end{pmatrix}.
                \]
                Therefore, we may assume without loss of generality that $A$ and $B$ represent quadruples of the same type.  Then $N=AB^{-1}$ is clearly an integer matrix.
                Finally, $N = AB^{-1}$ has determinant $+1$ since $A$ and $B$ are both positively ordered.
        \end{proof}

                In other words, the action of $\SO$ on the space of positively ordered Descartes quadruples of Gaussian circles is simply transitive.  Thus, such quadruples are in bijection with $\SO$ and hence with $\PGL_2(\ZZ[i])$.  Hence, such quadruples form a torsor for $\SO$, and the bijection depends on a choice of basepoint.  We will choose for the basepoint the \emph{base quadruple} in Figure \ref{fig:quadLVL}, given by 
        \[
                \begin{pmatrix} 0 & 0 & 2 & 2 \\
                                0 & 2 & 2 & 0 \\
                                0 & 0 & 2 & 0 \\
                               -1 & 1 & 1&1 
        \end{pmatrix}.
\]


        \begin{theorem}[Main Bijection, Second Part]
                \label{thm:mbsecond}
                Under the bijection of Theorem \ref{thm:mbfirst} (Main Bijection, First Part) between vertices of the Apollonian kingdom and Gaussian circles, chambers of the Apollonian kingdom are in bijection with Descartes quadruples of Gaussian circles.  In particular, two vertices are connected by an edge if and only if the two corresponding circles are tangent within some Descartes quadruple of Gaussian circles; and, components of the Apollonian kingdom are in bijection with Apollonian palaces.
        \end{theorem}

        \begin{proof}
Consider the first standard ultrabasis chamber with $u = (1,0)^T$, $v=(0,-i)^T$, whose vertices are
\[
        \begin{pmatrix}
                1 & 0 \\ 0 & 1 \end{pmatrix}
               , 
       \begin{pmatrix}
                -1 & i \\ i & 0 \end{pmatrix}
               , 
           \begin{pmatrix}
                0 & -i \\ -i & -1 \end{pmatrix}
               , 
       \begin{pmatrix}
                -1 & i \\ i-1 & i \end{pmatrix}.
       \]
       We will call this chamber the \emph{base chamber} (shown in Figure \ref{fig:chamberLVL}), because its associated circles (under the first part of the Main Bijection, Theorem \ref{thm:mbfirst}) form the base quadruple.  
      
The group $\PGLZ$ acts on the Apollonian kingdom in the obvious way, since it acts on lax vectors and lax lattices; this action is in agreement with the action of $\PGLZ$ on the corresponding circles via M\"obius transformations.  The image of a chamber is another chamber.  

Take an arbitrary chamber of the Apollonian kingdom, and express it in first standard ultrabasis form in terms of some $u$ and $v$.  Let $M = \begin{pmatrix} u & iv \end{pmatrix}$ be the matrix formed of columns $u$ and $iv$; it is an element of $\PGL_2(\ZZ[i])$.  Then this chamber is the image under $M$ of the base chamber, and so it must correspond to the quadruple which is the image under $\phi(M)$ of the base quadruple.  Hence the vertices of any chamber form a Descartes quadruple.  In other words, the Main Bijection
\[
        \mbox{vertices} \rightarrow \mbox{circles}
\]
induces a function from chambers to Descartes quadruples.  

This function must be injective because the Main Bijection is bijective.  The missing piece is surjectivity (i.e., perhaps the four vertices mapping to a quadruple aren't arranged into a chamber).  But any quadruple can be positively oriented and hence, by Proposition \ref{prop:so}, given as an image under $\SO$ of the base quadruple.  The corresponding element of $\PGLZ$ therefore maps the base chamber to another chamber whose circles form that quadruple. 
        \end{proof}

        \section{Afterthoughts}
        \label{sec:after}

        The action on circles and quadruples used in the proof of the Main Bijection is described by Graham, Lagarias, Mallows, Wilks and Yan in \cite{\gtone}, where they call it the \emph{M\"obius action}.  They also describe a \emph{right} action of $\SOR$ on the space of Descartes quadruples, an action which cannot be interpreted as acting on individual circles, but only on quadruples.  The subgroup of $\SO$ consisting of elements taking any quadruple to another in the same Apollonian packing (under this \emph{right} action) is called, appropriately, the \emph{Apollonian group}, originally defined by Hirst \cite{\Hirst}, and studied also in \cite{\AhSt, \ntone, \Soder}.  

\begin{theorem}
        Any primitive Descartes quadruple, upon dilation by a factor of two, can be realised as a quadruple of Gaussian circles.
\end{theorem}

\begin{proof}
        The \emph{right} action of $\SOR$ on Descartes quadruples acts on rows; the first row of a quadruple, which is formed of the four curvatures, is always an element of the light cone, by \eqref{eqn:descquad}.  This action of $\SOR$ on the light cone is transitive and defined over $\ZZ$; hence primitive vectors map to primitive vectors transitively (see \cite[Chapter 2.6]{KevinBrown}).  In particular, every quadruple of coprime integers $(b_1, b_2, b_3, b_4)$ is the image of $(0,0,1,1)$ under some element $N \in \SO$.  But then the image of the base quadruple under $N$ (acting on the right) is a quadruple of Gaussian circles having curvatures $2b_1, 2b_2, 2b_3, 2b_4$.
\end{proof}

This provides a new proof of Theorem 4.3 of \cite{\gttwo}, which states that any primitive Apollonian circle packing in $\CC_\infty$ can be made strongly integral by a Euclidean motion.

Somewhat analogously to the Apollonian group, the left action has a subgroup fixing the Apollonian strip packing of Figure \ref{fig:strip}; it may permute the other packings.

\begin{proposition}
        \label{prop:inftrees}
The subgroup of $\PGLZ$ taking any one Apollonian circle packing to itself is a conjugate of the subgroup
\[
        \Gamma = \left\langle 
\begin{pmatrix} 1 & 0 \\ 1 & 1 \end{pmatrix}, 
\begin{pmatrix} 0 & i \\ i & 1 \end{pmatrix}
        \right\rangle. 
\]
Furthermore, the images of $\RR$ under $\Gamma$ form the Apollonian strip packing (Figure \ref{fig:strip}) containing the base quadruple (Figure \ref{fig:quadLVL}).
\end{proposition}

\begin{proof}
The group is generated by the automorphisms of the base quadruple, together with a map which swaps the base quadruple with one of its neighbours as a chamber within the Apollonian kingdom.  
The swap is accomplished by
\[
A =         \begin{pmatrix}
                0 & -1 \\
                1 & 0 
        \end{pmatrix}.
\]
The automorphisms are generated by the two rotations
\[
B =        \begin{pmatrix}
                0 & i \\
                i & 1 
        \end{pmatrix},
C =         \begin{pmatrix}
                1 & -1 \\
                1 & 0 
        \end{pmatrix}
\]
around two vertices of a chamber.  If we let $D = C^{-1}A$, then there's a relation
\[
        A = B^{-1}DBD^{-1}B^{-1}DB.
\]
Therefore the group is given by $\langle D, B \rangle$.
\end{proof}

\begin{corollary}
        \label{cor:packing}
        The collection of images of the real line under the M\"obius transformations $N\Gamma$ for some $N \in \PGL_2(\ZZ[i])$ form a half-primitive strongly integral Apollonian circle packing.
\end{corollary}



If we take together all the left cosets $N\Gamma$ of $\Gamma$ for $N \in \PSL_2(\ZZ[i])$, drawing all the images of the real line, we obtain an Apollonian superpacking, as studied by Graham, Lagarias, Mallows, Wilks and Yan \cite{\gttwo}.  This picture, which when dilated by two those authors call the \emph{standard strongly integral super-packing}, is shown in red in Figure \ref{fig:superpacking}.  This repeating pattern contains one copy of every Apollonian circle packing, up to Euclidean motions, in its unit square (except the strip packing).


The Apollonian superpacking has a sister, which consists of the images of the real line under all $N\Gamma$ for $N \in \begin{pmatrix}i & 0 \\ 0 & 1 \end{pmatrix}\PSL_2(\ZZ[i])$, and is shown in blue in Figure \ref{fig:superpacking}.  Red and blue taken together, the resulting riot of circles is formed of all the images of the real line under $\PGL_2(\ZZ[i])$.  

\section{Appendix: Bestvina and Savin's topograph}
\label{sec:bestvina-savin}

Also inspired by Conway, but without reference to Apollonian circle packings, Bestvina and Savin study a complex formed of bases and superbases in $\ZZ[i]^2$ \cite{\bestsavin}.  Because a basis is contained in four superbases, each basis is associated to a square $2$-cell, whose edges are the superbases.  Square $2$-cells are glued three along each edge, reflecting the fact that a superbasis contains three bases.  The link of a vertex is the $1$-skeleton of a cube, which describes which superbases share common bases.

To compare Bestvina and Savin's complex to the Apollonian kingdom studied here, consider the six faces of the cubical link.  These correspond to vectors $u$, $v$, $w$, $u+iv$, $v+iw$, $w+iu$, the edges of a tetrahedron in the Apollonian kingdom.  The eight vertices of the cube can be broken up into two collections of four, each of which forms the vertices of a tetrahedron:  the collection of diagonals of each face which connect the four vertices gives the edges of the tetrahedron.

In this way, each vertex of Bestvina and Savin's complex gives rise to two tetrahedra.  An ultrabasis tetrahedron has a natural dual, since the vectors labelling the edges adjacent to a vertex of the tetrahedron also form a superbasis.  This corresponds to the dual Descartes quadruple described by Coxeter \cite{\coxeter}, and the pairing of tetrahedra at each vertex of Bestvina and Savin.


\bibliography{app-pack-bib}{}

\begin{thebibliography}{10}

\bibitem{MR1466797}
D.~Aharonov and K.~Stephenson.
\newblock Geometric sequences of discs in the {A}pollonian packing.
\newblock {\em Algebra i Analiz}, 9(3):104--140, 1997.

\bibitem{MR2914631}
Mladen Bestvina and Gordan Savin.
\newblock Geometry of integral binary {H}ermitian forms.
\newblock {\em J. Algebra}, 360:1--20, 2012.

\bibitem{MR2813334}
Jean Bourgain and Elena Fuchs.
\newblock A proof of the positive density conjecture for integer {A}pollonian
  circle packings.
\newblock {\em J. Amer. Math. Soc.}, 24(4):945--967, 2011.

\bibitem{1205.4416}
Jean Bourgain and Alex Kontorovich.
\newblock On the strong density conjecture for integral apollonian circle
  packings, 2012.
\newblock \href{http://arxiv.org/abs/1205.4416}{\texttt{arXiv:1205.4416}}.

\bibitem{KevinBrown}
Kevin Brown.
\newblock {\em Reflections on Relativity}.
\newblock lulu.com, 2012.

\bibitem{IPE}
Otfried Cheong.
\newblock {\em The {I}pe {E}xtensible {D}rawing {E}ditor ({V}ersion 7.1.1)},
  2012.
\newblock \href{http://ipe7.sourceforge.net}{{\tt
  http://ipe7.sourceforge.net/}}.

\bibitem{MR1478672}
John~H. Conway.
\newblock {\em The sensual (quadratic) form}, volume~26 of {\em Carus
  Mathematical Monographs}.
\newblock Mathematical Association of America, Washington, DC, 1997.
\newblock With the assistance of Francis Y. C. Fung.

\bibitem{MR0230204}
H.~S.~M. Coxeter.
\newblock The problem of {A}pollonius.
\newblock {\em Amer. Math. Monthly}, 75:5--15, 1968.

\bibitem{Descartes}
R.~Descartes.
\newblock {\em Oeuvres de Descartes}, volume~4.
\newblock Charles Adam \& Paul Tannery, 1901.

\bibitem{MR3020827}
Elena Fuchs.
\newblock Counting problems in {A}pollonian packings.
\newblock {\em Bull. Amer. Math. Soc. (N.S.)}, 50(2):229--266, 2013.

\bibitem{MR1971245}
Ronald~L. Graham, Jeffrey~C. Lagarias, Colin~L. Mallows, Allan~R. Wilks, and
  Catherine~H. Yan.
\newblock Apollonian circle packings: number theory.
\newblock {\em J. Number Theory}, 100(1):1--45, 2003.

\bibitem{MR2173929}
Ronald~L. Graham, Jeffrey~C. Lagarias, Colin~L. Mallows, Allan~R. Wilks, and
  Catherine~H. Yan.
\newblock Apollonian circle packings: geometry and group theory. {I}. {T}he
  {A}pollonian group.
\newblock {\em Discrete Comput. Geom.}, 34(4):547--585, 2005.

\bibitem{MR2183489}
Ronald~L. Graham, Jeffrey~C. Lagarias, Colin~L. Mallows, Allan~R. Wilks, and
  Catherine~H. Yan.
\newblock Apollonian circle packings: geometry and group theory. {II}.
  {S}uper-{A}pollonian group and integral packings.
\newblock {\em Discrete Comput. Geom.}, 35(1):1--36, 2006.

\bibitem{MR0209981}
K.~E. Hirst.
\newblock The {A}pollonian packing of circles.
\newblock {\em J. London Math. Soc.}, 42:281--291, 1967.

\bibitem{kocikminkowski}
Jerzy Kocik.
\newblock A theorem on circle configuarations, 2007.
\newblock \href{http://arxiv.org/abs/0706.0372}{\texttt{arXiv:0706.0372}}.

\bibitem{MR2784325}
Alex Kontorovich and Hee Oh.
\newblock Apollonian circle packings and closed horospheres on hyperbolic
  3-manifolds.
\newblock {\em J. Amer. Math. Soc.}, 24(3):603--648, 2011.
\newblock With an appendix by Oh and Nimish Shah.

\bibitem{MR1903421}
Jeffrey~C. Lagarias, Colin~L. Mallows, and Allan~R. Wilks.
\newblock Beyond the {D}escartes circle theorem.
\newblock {\em Amer. Math. Monthly}, 109(4):338--361, 2002.

\bibitem{Lionel}
L.~{Levine}, W.~{Pegden}, and C.~K. {Smart}.
\newblock {Apollonian Structure in the Abelian Sandpile}, 2012.
\newblock \href{http://arxiv.org/abs/1208.4839}{\texttt{arXiv:1208.4839}}.

\bibitem{Lionel2}
L.~{Levine}, W.~{Pegden}, and C.~K. {Smart}.
\newblock {The Apollonian structure of integer superharmonic matrices}, 2013.
\newblock \href{http://arxiv.org/abs/1309.3267}{\texttt{arXiv:1309.3267}}.

\bibitem{MR2250043}
S.~Northshield.
\newblock On integral {A}pollonian circle packings.
\newblock {\em J. Number Theory}, 119(2):171--193, 2006.

\bibitem{MR1017034}
Dan Pedoe.
\newblock {\em Geometry}.
\newblock Dover Books on Advanced Mathematics. Dover Publications Inc., New
  York, second edition, 1988.
\newblock A comprehensive course.

\bibitem{MR0215169}
Daniel Pedoe.
\newblock On a theorem in geometry.
\newblock {\em Amer. Math. Monthly}, 74:627--640, 1967.

\bibitem{MR776784}
Roger Penrose and Wolfgang Rindler.
\newblock {\em Spinors and space-time. {V}ol. 1}.
\newblock Cambridge Monographs on Mathematical Physics. Cambridge University
  Press, Cambridge, 1984.
\newblock Two-spinor calculus and relativistic fields.

\bibitem{SarnakLetter}
Peter Sarnak.
\newblock Letter to {L}agarias.
\newblock \href{http://www.math.princeton.edu/sarnak}{{\tt
  http://www.math.princeton.edu/sarnak}}.

\bibitem{MR2800340}
Peter Sarnak.
\newblock Integral {A}pollonian packings.
\newblock {\em Amer. Math. Monthly}, 118(4):291--306, 2011.

\bibitem{MR1178634}
Bo~S{\"o}derberg.
\newblock Apollonian tiling, the {L}orentz group, and regular trees.
\newblock {\em Phys. Rev. A (3)}, 46(4):1859--1866, 1992.

\bibitem{Sage}
W.\thinspace{}A. Stein et~al.
\newblock {\em {S}age {M}athematics {S}oftware ({V}ersion 4.8)}.
\newblock The Sage Development Team, 2012.
\newblock \href{http://www.sagemath.org}{{\tt http://www.sagemath.org}}.

\end{thebibliography}
\bibliographystyle{plain}

\end{document}